\documentclass[11pt,reqno]{amsart}
\usepackage{amssymb,amsthm,amsmath,hyperref}
\numberwithin{equation}{section}
\usepackage{cases}
\usepackage{xcolor}
  \usepackage{paralist}
  \usepackage{graphics} 
  \usepackage{epsfig}  
\usepackage{graphicx}  \usepackage{epstopdf} 
\numberwithin{equation}{section}
\newtheorem{thm}{Theorem}[section]
\newtheorem{lem}{Lemma}[section]
\newtheorem{claim}{Claim}[section]
\newtheorem{cor}{Corollary}[section]
\newtheorem{prop}{Proposition}[section]
\theoremstyle{definition}
\newtheorem{defn}{Definition}[section]
\newtheorem{con1}{Conjecture}[section]
\theoremstyle{remark}
\newtheorem{rem}{Remark}[section]
\allowdisplaybreaks

\usepackage[a4paper, left = 3.2cm, right = 3.2cm, top = 2.54cm, bottom = 2.54cm]{geometry}
\begin{document}
\title[Mass-critical inhomogeneous NLS ]{Global well-posedness and scattering for    mass-critical inhomogeneous NLS when  $d\ge3$}
 \author[Liu]{Xuan Liu}
\address{ School of Mathematics, Hangzhou Normal   University, Hangzhou 311121, China}
\email{liuxuan95@hznu.edu.cn }
 \author[Miao]{Changxing Miao}
\address{ Institute of Applied Physics and Computational Mathematics, Beijing 100088, China
\newline\indent
National Key Laboratory of Computational Physics, Beijing 100088, China}
\email{miao\_changxing@iapcm.ac.cn }
 \author[Zheng]{Jiqiang Zheng}
\address{Institute of Applied Physics and Computational Mathematics, Beijing 100088, China
\newline\indent
National Key Laboratory of Computational Physics, Beijing 100088, China}
\email{zhengjiqiang@gmail.com, zheng\_jiqiang@iapcm.ac.cn}
\date{}
\maketitle
\begin{abstract}
	We prove global well-posedness and scattering for solutions to the mass-critical inhomogeneous nonlinear Schr\"odinger equation  $i\partial_{t}u+\Delta u=\pm |x|^{-b}|u|^{\frac{4-2b}{d}}u$ for large  $L^2(\mathbb{R} ^d)$ initial data with  $d\ge3,0<b<\min \left\{ 2,\frac{d}{2} \right\}$; in the focusing case, we   require that the mass is strictly less than that of the ground state.
	
	Compared with the classical Schr\"odinger case ($b=0$, Dodson, J. Amer. Math. Soc. (2012),  Adv. Math. (2015)), the  main differences for the inhomogeneous case ($b>0$) are that the presence of the inhomogeneity  $|x|^{-b}$ creates a nontrivial singularity at the origin, and    breaks the translation symmetry  as well as the Galilean invariance of the equation, which makes the establishment of the   profile decomposition and long time Strichartz estimates more difficult.  To overcome these difficulties, we perform the concentration compactness/rigidity  methods of [Kenig and Merle, Invent. Math. (2006)] in the Lorentz space framework, and   reduces the problem to the exclusion of almost periodic solutions. The exclusion of these solutions will utilize fractional estimates  and  long time Strichartz estimates in Lorentz spaces.

	 In our study, we obseve that the decay of the inhomogeneity  $|x|^{-b}$ at infinity prevents the concentration of  the almost periodic solution   at infinity in either physical or frequency space.  Therefore, we can use classical Morawetz estimates, rather than interaction Morawetz estimates, to exclude the existence of the quasi-soliton.   Moreover, our  method can also be  applied  to the case \( b = 0 \), thereby recovering the scattering results in [Tao, Visan and Zhang, Duke Math. J. (2007), Killip, Visan and Zhang, Anal. PDE (2008)].  
\end{abstract}
\section{Introduction}\label{s1}
We consider the following  mass-critical inhomogeneous  nonlinear Schr\"odinger equation 
\begin{equation}\label{nls}
	\begin{cases}
		i\partial_{t}u+\Delta u=\mu |x|^{-b}|u|^{ \frac{4-2b}{d}} u,\qquad (t,x)\in \mathbb{R} \times \mathbb{R} ^d\\
		u|_{t=0}=u_0\in   L^2(\mathbb{R}^d),
	\end{cases}
\end{equation}
where  $d\ge3,0<b<\min \left\{ 2,\frac{d}{2} \right\}$.         Here  $\mu=\pm1$, with  $\mu=+1$ known as the \emph{defocusing} case and  $\mu=-1$ as the \emph{focusing} case. 
This model arises in the setting of nonlinear optics, where the factor  $|x|^{-b}$  represents some inhomogeneity in the medium (see, e.g., \cite{Gill2000,LT1994}).  As pointed out by Genoud and Stuart \cite{GS2008},  the factor  $|x|^{-b}$  appears naturally as a limiting case of potentials that decay polynomially at infinity.

In this paper, we study the global well-posedness and scattering  of the Cauchy problem (\ref{nls}). Before stating our results, we give the definition of the solutions and the  corresponding notion of blowup.   
\begin{defn}[Strong solution]
	Let  $F(u) =\pm |x|^{-b}|u|^ \frac{4-2b}{d} u$ and 
	\begin{equation}
			\gamma=:\frac{2d+8-4b}{d},\qquad \rho=: \frac{2d+8-4b}{d+2-2b}. \label{E:gamma}
	\end{equation}
	A function $u : I\times \mathbb{R}^d \to \mathbb{C}$  is a solution to (\ref{nls}) if for any compact  $J\subset I$,  $u\in C_t^0L^2_x(J\times \mathbb{R} ^d)\cap L^\gamma _tL_x^{\rho,2}(J\times \mathbb{R} ^d)$,  and we have the Duhamel formula for all $t, t_0 \in I$:
\begin{equation}
		u(t) = e^{i(t - t_0)\Delta} u(t_0) - i \int_{t_0}^{t} e^{i(t-\tau)\Delta} F(u(\tau))  d\tau,\notag
\end{equation}
where  $L_x^{\rho,2}$ is the Lorentz space defined in Subsection \ref{S:2.2}  below.   We refer to the interval $I$ as the \textit{lifespan} of $u$. We say that $u$ is a \textit{maximal-lifespan solution} if the solution cannot be extended to any strictly larger interval. We say that $u$ is a \textit{global solution} if $I = \mathbb{R}$.
\end{defn}
\begin{defn}[Blow up]
We say that a solution   $u$ to (\ref{nls}) blows up forward in time if there exists a time   $t_0\in I$ such that
\begin{equation}
		\|u\|_{L_t^\gamma  L_x^{\rho,2}([t_0, \sup(I)) \times \mathbb{R}^d)} = \infty,\notag
\end{equation}
	and that   $u$ blows up backward in time if there exists a time   $t_0\in I$ such that
\begin{equation}
		\|u\|_{L_t^\gamma  L_x^{\rho,2}((\inf(I), t_0] \times \mathbb{R}^d)} = \infty.\notag
\end{equation}
\end{defn}
 
On the interval of existence, the solution preserves its mass  (see  Proposition \ref{T:CP})
\begin{equation}
	M(u) =: \int_{\mathbb{R}^d}  |u(t,x)|^2\,dx=M(u_0).\label{E:mass}
\end{equation}
Equation (\ref{nls}) is referred to as  mass-critical as the natural scaling of the equation  $u(t,x)\mapsto \lambda ^{\frac{d}{2}} u(\lambda^2 t,\lambda x)$  also keeps the  mass invariant.
 
By considering Strichartz estimates in Lorentz spaces, Aloui and Tayachi \cite{Aloui} proved the  $L^2(\mathbb{R} ^d)$ local well-posedness   for the mass-critical   inhomogeneous NLS (\ref{nls}) (see  Proposition \ref{T:CP}).  In this paper, we consider the questions of global well-posedness and scattering of the solutions.   Here scattering refers to the fact that there exists unique  $u_{\pm}\in L^2(\mathbb{R} ^d)$  such that  
\begin{equation}
  \lim_{t\to\pm\infty}\|u(t)-e^{it\Delta}u_\pm\|_{L^2} = 0.\label{7151}
\end{equation}
 Such results have been extensivlely studied   for the classical mass-critical Schr\"odinger equation (\cite{Dodson4,Dodson3,Dodson2,KTV,KVZ,TVZ}), and energy-critical   Schr\"odinger equation (\cite{Bourgain1999,Colliander2008,Dodson2019,KillipVisan2010,Ryckman2007,Tao2006,Visan2007}).

To be more precise,  let us first recall the global well-posedness and scattering results for the classical mass-critical  Schr\"odinger   equation ($b=0$):
\begin{equation}
\begin{cases}
		i\partial_{t}u+\Delta u=\mu |u|^{\frac{4}{d}}u,\qquad (t,x)\in \mathbb{R} \times \mathbb{R} ^d \\
		u(0,x)=u_0(x)\in L^2(\mathbb{R} ^d).
\end{cases}\label{E:cnls}
\end{equation}
It is well  known  that all small mass solutions are global and scatte; while in the focusing case this assertion can fail for solutions with large mass.  In fact, let $ \widetilde{Q}:\mathbb{R} ^d\rightarrow \mathbb{R} ^+ $ be the unique positive  radial  solution to the elliptic equation
\begin{equation}
	\Delta \widetilde{Q}+\widetilde{Q}^{1+\frac{4}{d}}=\widetilde{Q},\notag
\end{equation}
then  $u(t,x)=:e^{it}\widetilde{Q}(x)$, with mass  $M(\widetilde{Q})$,  is a  nonscattering solution to (\ref{E:cnls}).  However, it is widely believed that   $M(\widetilde{Q})$ is the minimal mass obstruction to global well-posedness and scattering in the focusing case, and that no such obstruction exists in the defocusing case:
\begin{con1}[Global existence and scattering]\label{CJ}
	\label{conjecture:global-existence-scattering}
	Let $d \geq 1$, $\mu \in \{-1,+1\}$ and  $u_0\in L^2(\mathbb{R} ^d)$.  In the focusing case, we further assume   that   $M(u_0)<M(\widetilde{Q})$. Then all maximal-lifespan solutions to  (\ref{E:cnls}) are globally well-posed and scatter.         
\end{con1}
Conjecture \ref{CJ} has now been completely resolved. 
 Killip, Tao and  Visan \cite{KTV} first solved Conjecture \ref{CJ} for    radial    initial data in dimension  $d=2$. 
  They introduced the concept of almost periodic solutions,  and under the assumption that Conjecture \ref{CJ} is not valid, they   employed the  concentration compactness/rigidity method of \cite{Kenig2006}  to construct an almost periodic solution.   Then they categorized the almost periodic solutions into three cases: soliton-like solution, double high-to-low frequency cascade, self-similar solution. Finally, they  completed the proof by excluding these three types of solutions. 
 Later,   Tao,   Visan and  Zhang \cite{TVZ} used frequency localized Morawetz estimates to exclude the existence of almost periodic solutions, thus solving Conjecture  \ref{CJ}  for  $\mu=1,d\ge3$, and radial    initial data. 
By applying fractional chain rules to overcome the  obstacle of the fractional nonlinearity and establishing a Gronwall-type inequality to deal with the frequency relations between the solution and the nonlinearity,   Killip,  Visan and  Zhang  \cite{KVZ} solved   Conjecture  \ref{CJ} for  $\mu=-1,d\ge3$.

Recently,  Dodson established long-time Strichartz estimates, combining    with interaction Morawetz estimates,  and through detailed  analysis on  the frequency scale function, completedy solved Conjecture  \ref{CJ}  for  $ d\ge1$,  and  general    initial data.    See \cite{Dodson1,Dodson3,Dodson2} for the defocusing case, and   \cite{Dodson4} for the focusing case. We list the above results in Table 1 and 2 as follows.

\begin{table}[h]
	\centering
	\caption{Defocusing case: $\mu=+1$}
	\begin{tabular}{|c|c|c|c|}
		\hline
		& $d=1$ & $d=2$ & $d \ge 3$ \\ \hline
		radial  & &Killip-Tao-Visan \cite{KTV}& Tao-Visan-Zhang  \cite{TVZ}\\ \hline
		 non radial & Dodson \cite{Dodson2} & Dodson \cite{Dodson3} & Dodson \cite{Dodson1}\\ \hline
	\end{tabular}
\end{table}
\begin{table}[h]
	\centering
	\caption{Focusing case: $\mu=-1$}
	\begin{tabular}{|c|c|c|c|}
		\hline
		& $d=1$ & $d=2$ & $d \ge 3$ \\ \hline
	 radial & &Killip-Tao-Visan \cite{KTV} & Killip-Visan-Zhang \cite{KVZ} \\ \hline
		non radial  & Dodson  \cite{Dodson4}& Dodson \cite{Dodson4}&Dodson \cite{Dodson4}\\ \hline
	\end{tabular}
\end{table}

Inspired by the work of \cite{Dodson4,Dodson3,Dodson2,KTV,KVZ,TVZ},   our aim of  this paper is to extend their results ($b=0$) to the inhomogeneous case ($b\neq0$).  
According to  the local well-posedness theory (see  Proposition  \ref{T:CP}),    all maximal-lifespan solutions to (\ref{nls}) with sufficiently small mass are  global and  scatter.  However, in the focusing case where  $\mu = -1$, it is well known that this assertion can fail for solutions with large mass. Let  $Q$ be the unique positive solution to 
\begin{equation}
	\Delta Q+|x|^{-b}Q^{\frac{4-2b}{d}+1}=Q.\label{E:ground state} 
\end{equation} 
The existence of  a  positive solution  to (\ref{E:ground state}) was proved in \cite{Genoud2008, GS2008}  for  dimension   $d\ge2$,  and for   $d=1$ in \cite{Genoud2010}. The issue of uniqueness  was addressed for dimension  $d\ge3$ by Yanagida \cite{Yanagida1991} (see also \cite{Genoud2008}),   for dimension  $d=2$ by   \cite{Genoud2011}, and for dimension   $d=1$ by  \cite{Toland1984}.
The unique positive solution  $Q$ is called the \emph{ground state}, and provides a stationary solution 
\begin{equation}
	u(t,x)=e^{it}Q(x)\notag
\end{equation} 
to (\ref{nls}) which blows up at infinity both forward and backward in time. 

Similar to   the case of the   mass-critical Schr\"odinger equation (\ref{E:cnls}),   it is  conjectured that  $M(Q)$ is the minimal mass obstruction to global well-posedness and scattering in the focusing case, and that no such obstruction exists in the defocusing case. In fact, there is compelling evidence supporting this conjecture when the initial data  $u_0$ possesses additional regularity.  For the defocusing equation, the global well-posedness for initial data in  $H^1_x(\mathbb{R} ^d)$ can be established by the contraction argument and the conservation of  mass and  energy, see e.g. \cite{Cazenave}.  Recall that the energy is given by
\begin{equation}
	E(u) =: \int_{\mathbb{R}^d}\left( \frac12|\nabla u(t,x)|^2 +\frac{d\mu }{2(d+2-b)} |x|^{-b} |u(t,x)|^{\frac{4-2b}{d} +2}\right)\,dx.\label{E:energy}
\end{equation}
The focusing equation with data in   $H^1_x$ was treated by Genoud \cite{Genoud2012}.   
By variational methods,  Genoud \cite{Genoud2012}  established the sharp Gagliardo-Nirenberg's inequality 
	\begin{equation}
	\int_{\mathbb{R} ^d} |x|^{-b}|f(x)|^{\frac{4-2b}{d} +2}dx\le \frac{d+2-b}{d}\left(\frac{ \|f\|_{L^2} }{ \|Q\|_{L^2} }\right)^{\frac{4-2b}{d}}\int _{\mathbb{R} ^d}|\nabla f(x)|^2dx. \label{E:GN1}
\end{equation}
 When  $u_0\in H^1$ and $\|u_0\|_{L^2} < \|Q\|_{L^2}$, this  inequality establishes an upper bound of  the 
 $H^1$ norm of the solution at all times during its existence.  Then  Genoud \cite{Genoud2012} proved global well-posedness of the focusing equation for initial data in  $H^1$ with mass less than that of the ground state. As    $e^{it}Q$ is a blow-up solution with mass   $ \|Q\|_{L^2}$, we can reasonably conjecture that $ \|Q\|_{L^2}$  is the scattering threshold  for the focusing equation.

 In this paper, we prove that this conjecture holds for $d\ge3$.  
 \begin{thm}\label{T:1}
 	Let $d\ge3$,  $0<b<\min \left\{2,\frac{d}{2}\right\} $ and  $\mu\in \left\{-1,+1 \right\}$.  In the focusing case,  we further assume   that  $M(u_0)<M(Q)$.  
 Then all maximal-lifespan solutions to (\ref{nls}) are global and scatter  in the sense (\ref{7151}).  
 \end{thm}

In the proof of  Theorem \ref{T:1},  we perform    the  concentration compactness/rigidity  arguments in \cite{Kenig2006},  and   apply the the long time Strichartz estimates and frequency localized  Morawetz estimates techniques  of \cite{Dodson1,Dodson4}) within the framework of Lorentz spaces. 
 Before we state our argument, we need some definitions. 
\begin{defn}
	\text{(Symmetry group)}. For any phase $\theta \in \mathbb{R}/2\pi \mathbb{Z}$ and scaling parameter $\lambda > 0$, we define the unitary transformation $g_{\theta,   \lambda} : L_x^2(\mathbb{R}^d) \to L_x^2(\mathbb{R}^d)$ by the formula:
	\[
	[g_{\theta,  \lambda} f](x) := \frac{1}{\lambda^{d/2}} e^{i \theta}  f\left( \frac{x  }{\lambda} \right).
	\]
	We let $G$ be the collection of such transformations.  
\end{defn}
\begin{defn}\label{D:1}
	\text{(Almost periodicity modulo symmetries)}. A solution $u$ with lifespan $I$ is said to be \emph{almost periodic modulo $G$} if there exist  functions $N : I \to \mathbb{R}^+$  and a function $C : \mathbb{R}^+ \to \mathbb{R}^+$ such that
\begin{equation}
	\int _{|x|\ge C(\eta)/N(t)}|u(t,x)|^2dx\le \eta,\label{E:compact1}
\end{equation}   
and   
\begin{equation}
	\int_{|\xi|\ge C(\eta)N(t)}|\hat u(t,\xi)|^2d\xi\le \eta,\label{E:compact2}
\end{equation}
	for all $t \in I$ and $\eta > 0$. We refer to the function $N(t)$ as the \textit{frequency scale function} for the solution $u$,  and $C(\eta)$ as the \textit{compactness modulus function}.  
\end{defn}
\begin{rem}\label{REM}
	Compared with the definition of almost periodic solutions of the classical Schr\"odinger equation (see e.g. [Definition 1.14, \cite{KTV}]), there is no spatial center function   $x(t)$ and frequency center function  $\xi(t)$  in  (\ref{E:compact1}) and (\ref{E:compact2}).  
This comes from the fact that the inhomogeneity  $|x|^{-b}$  decays at infinity, which  prevents the concentration of  the almost periodic solution   at infinity in either physical or frequency space (see  Proposition \ref{P:embed} below).  	Therefore, the spatial and frequency translation functions  must be bounded, and by a    translation we can let $x(t)\equiv\xi(t)\equiv0$.   For the details, we refer to    Section \ref{S:3}.   
\end{rem}
Assuming for contradiction that Theorem \ref{T:1} is false,  we can identify an almost periodic solution. This solution must take one of the following three scenario:
\begin{thm}[Three special scenarios for blowup]\label{T:reduction} Suppose  Theorem~\ref{T:1} fails,  then there exists a maximal-lifespan solution   to (\ref{nls}), which is also almost periodic modulo  $G$.  Moreover, it blows up both forward and backward in time,  	 and in the focusing case also obeys  $M(u)<M(Q)$.

Furthermore, we can also ensure that the lifespan $I$ and the frequency scale function $N : I \to \mathbb{R}^+$ match one of the following three scenarios:
	\begin{enumerate}
		\item \text{(Soliton-like solution)} We have $I = \mathbb{R}$ and
		\[
		N(t) = 1 \quad \text{for all } t \in \mathbb{R}.
		\]		
		\item \text{(Double high-to-low frequency cascade)} We have $I = \mathbb{R}$,
		\[
		\liminf_{t \to -\infty} N(t) = \liminf_{t \to +\infty} N(t) = 0,
		\]
		and
		\[
		\sup_{t \in \mathbb{R}} N(t)\le1\quad \text{for all } t \in I.
		\]
		
		\item \text{(Self-similar solution)} We have $I = (0, +\infty)$ and
		\[
		N(t) = t^{-1/2} \quad \text{for all } t \in I.
		\]
	\end{enumerate}
\end{thm}
  Theorem \ref{T:reduction}  will be  proved  in Section \ref{S:3}.    Therefore, to prove Theorem \ref{T:1}, it is sufficient to demonstrate that an almost periodic solution to (\ref{nls}), which  satisfies  the properties outlined in Theorem \ref{T:reduction}, cannot exist.  By a time translation and possibly modifying the  $C(\eta)$  in (\ref{E:compact1})--(\ref{E:compact2}), we can assume that  $N(0)=1$ and  $N(t)\le1 $ for all  $t>0$.   
 Following the arguments in \cite{Dodson1,Dodson4,Dodson3,Dodson2}, it is sufficient to rule out 
 the rapid frequency cascade
 \begin{equation}
\int_0^\infty N(t)^3dt<\infty ,\label{E:rapid}
 \end{equation}
 and the quasi-soliton 
 \begin{equation}
   \int_0^\infty N(t)^3dt=\infty. \label{E:Quasi}
 \end{equation}

The main  ingredient used to exclude scenariors (\ref{E:rapid}) and (\ref{E:Quasi}) is the following long time Strichartz estimate, which is proved by  combining an induction on frequency estimate with the bilinear Strichartz estimate.
\begin{prop}\label{T:long time SZ}
	Suppose  $J\subset [0,\infty )$ is compact,  $d\ge3$, $0<b<\min \left\{2,\frac{d}{2}\right\} $,  $u$ is a minimal mass blow up solutions to (\ref{nls}) in the form of Theorem \ref{T:reduction},  $\int _JN(t)^3dt=K$. 
	Then 
		\begin{equation}
		\|P_{>N}u\|_{L^2_tL_x^{\frac{2d}{d-2},2}(J\times \mathbb{R} ^d)}\lesssim   \big(\tfrac{K}{N}\big)^{1/2}+\sigma _{J}\big(\tfrac{N}{2}\big),\label{long time 1}
	\end{equation} 
	where for  $N=2^k, k\in \mathbb {Z}$, the frequency envelopes 
	\begin{equation}
		\sigma _J(N)=:  \sum_{j=-\infty}^{\infty} 2^{-|j-k|/3} \inf_{t \in J} \left\| P_{>2^j} u(t) \right\|_{L_x^2(\mathbb{R}^d)}.\label{691}
	\end{equation}    	
	More generally, suppose  $\frac{1}{2}\le s< \min \left\{ \frac{d-2b}{2}+1,\frac{4-2b}{d}+1 \right\}$, then 
	\begin{equation}
			\|P_{>N}u \|_{L^2_tL_x^{\frac{2d}{d-2},2}(J\times \mathbb{R} ^d)}\lesssim_s \frac{K^{\frac{1}{2}}}{N^{s}} \|u\|_{L^\infty _t \dot H_x^{s-\frac{1}{2}}(J\times \mathbb{R} ^d)} +\sigma _J\big(\tfrac{N}{2}\big).\label{long time 2}
	\end{equation} 
\end{prop}
The above long time Strichartz estimate will be used to prove that if  $u(t,x)$ is an almost periodic solution to  (\ref{nls}) and  $\int_0^\infty N(t)^3dt<\infty $, then  $u(t,x)$ possesses additional regularity. Combining this regularity with the conservation of the energy, we can  preclude the possibility of the rapid frequency cascade (\ref{E:rapid}).  For the details, we refer to Section \ref{S:5}.

To preclude  the quasi-soliton (\ref{E:Quasi}), we utilize a frequency truncated   Morawetz estimate as in \cite{Dodson4}. These estimates scale like  $\int_0^TN(t)^3dt$ and are actually bounded below by a constant times   $\int_0^TN(t)^3dt$.  On the other hand, by the compactness (\ref{E:compact2}), these estimates can also be bounded above by a small constant times   $\int_0^TN(t)^3dt$. Therefore, we obtain a contradiction by  letting  $T$ sufficiently large. In this process, we need to  truncate to the low frequencies since the Morawetz action scales like  $ \|u(t)\|_{\dot H^{1/2}_x}^2 $. Then the long time Strichartz estimate (\ref{long time 1}) allows us to control the errors that arise from  frequency truncation.   Therefore,   the combinition of  the frequency truncated   Morawetz estimate and the long time Strichartz estimate allows us to exclude the existence of  quasi-soliton (\ref{E:Quasi}), thereby completing the proof of Theorem \ref{T:1}.   For the details, we refer to Section \ref{S:6}.

 Compared the inhomogeneous Schr\"odinger equation (\ref{nls}) with the classical mass-critical Schr\"odinger equation  (\ref{E:cnls}), the  main differences for the inhomogeneous case ($b>0$) are that the presence of the inhomogeneity  $|x|^{-b}$ creates a nontrivial singularity at the origin, and    breaks the translation symmetry  as well as the Galilean invariance of the equation, which makes the establishment of the   profile decomposition and long time Strichartz estimates more difficult.   To address these difficulties, we adopt the approach in \cite{Aloui} and  perform the arguments of \cite{Dodson1,Dodson4,Kenig2006}) within the framework of Lorentz spaces. 
 On the other hand, the presence of \(|x|^{-b}\) also brings some advantages. In fact, as noted in  Remark \ref{REM}, the decay of \(|x|^{-b}\) at infinity prevents the concentration of  the almost periodic solution   at infinity in either physical or frequency space.   Therefore, the spatial and frequency translation functions  must be bounded, and by a    translation we can let $x(t)\equiv\xi(t)\equiv0$. 
 This  effectively places us in the same scenario as in the radial case but without the need for a radial assumption.
 As a benefit of this, we can use classical Morawetz estimates, rather than interaction Morawetz estimates, to exclude the existence of the quasi-soliton.

Finally, we would like to point out that,  with only trivial modifications,  the proof presented in our paper can also be  applied  to the case \( b = 0 \), thereby recovering the scattering results in  \cite{TVZ,KVZ}.

The rest of the paper is  organized as follows:
In Section \ref{S:2}, we  record some harmonic analysis tools in Lorentz spaces,  recall well-posedness and the stability results of (\ref{nls}),   and construct scattering solutions away from the origin. 
In Section \ref{S:3}, we demonstrate that if Theorem \ref{T:1} fails, we can  find the almost periodic solutions as  described in Theorem \ref{T:reduction}.  In Section \ref{S:4}, we  prove the long time Strichartz estimates.  
In Section \ref{S:5}, we use the  long time Strichartz estimates to  rule out the rapid cascade solutions.
In Section \ref{S:6}, we preclude the quasi-soliton solutions, thereby completing the proof of Theorem \ref{T:1}. 
Finally, in the Appendix,   we provide the construction of  $N_m(t)$ and  present a detailed proof of the frequency error estimates  in Section    \ref{S:6}. 
\section{Preliminaries}\label{S:2}
\subsection{Some notations}
We use the standard notation for mixed Lebesgue space-time norms and Sobolev spaces. 
We write  $A\lesssim  B$ to denote  $A\le CB$ for some  $C>0$. If  $A\lesssim B$ and  $B\lesssim A$, then we write  $A\approx B$.  We write  $A\ll B$   to denote  $A\le cB$ for some  small $c>0$.  If  $C$ depends upon some additional parameters, we will indicate this with subscripts; for example,  $X\lesssim _u Y$ denotes that  $X\le C_u Y$ for some  $C_u$ depending on  $u$.    We use  $O(Y)$ to denote any quantity  $X$  such that  $|X|\lesssim  Y$. We use     $\langle x \rangle $ to denote  $(1+|x|^2)^{\frac{1}{2}}$.        We write $L_t^q L_x^r$ to denote the Banach space with norm
\[
\|u\|_{L_t^q L_x^r (\mathbb{R} \times \mathbb{R}^d)} := \left( \int_{\mathbb{R}} \left( \int_{\mathbb{R}^d} |u(t, x)|^r \, dx \right)^{q/r} \, dt \right)^{1/q},
\]
with the usual modifications when $q$ or $r$ are equal to infinity, or when the domain $\mathbb{R} \times \mathbb{R}^d$ is replaced by spacetime slab such as $I \times \mathbb{R}^d$. When $q = r$ we abbreviate $L_t^q L_x^q$ as $L_{t,x}^q$. 

Let $\phi(\xi)$ be a radial bump function supported in the ball $\{\xi \in \mathbb{R}^d : |\xi| \leq 2\}$ and equal to  $1$ on the ball $\{\xi \in \mathbb{R}^d : |\xi| \leq 1\}$. For each number $N > 0$, we define the Fourier multipliers
\begin{equation}
	\widehat{P_{\leq N} f}(\xi) =: \phi(\xi/N) \widehat{f}(\xi), \quad \widehat{P_{> N} f}(\xi) =: (1 - \phi(\xi/N)) \widehat{f}(\xi),\notag
\end{equation}
\begin{equation}
	\widehat{P_N f}(\xi) = :\psi(\xi/N) \widehat{f}(\xi) =: (\phi(\xi/N) - \phi(2\xi/N)) \widehat{f}(\xi).\notag
\end{equation}
We similarly define $P_{< N}$ and $P_{\geq N}$.  For convenience of notation, let $u_N =: P_{N}u$, $u_{\leq N} =: P_{\leq N}u$, and $u_{>N} =: P_{>N}u$.
We will normally use these multipliers when $M$ and $N$ are dyadic numbers (that is, of the form $2^n$ for some integer $n$); in particular, all summations over $N$ or $M$ are understood to be over dyadic numbers.  
 
\subsection{Lorentz spaces and  nonlinear estimates}\label{S:2.2}
Let $f$ be a measurable function on $\mathbb{R}^d$. The distribution function of $f$ is defined by
\begin{equation}
	d_f(\lambda)=: |\{x\in \mathbb{R}^d : |f(x)|>\lambda\}|, \quad \lambda>0,\notag
\end{equation}
where $|A|$ is the Lebesgue measure of a set $A$ in $\mathbb{R}^d$. The decreasing rearrangement of $f$ is defined by
\begin{equation}
	f^*(s)=: \inf \left\{ \lambda>0 : d_f(\lambda)\leq s\right\}, \quad s>0.\notag
\end{equation}

\begin{defn}[Lorentz spaces] ~\\
	Let $0<r<\infty$ and $0<\rho\leq \infty$. The Lorentz space $L^{r,\rho}(\mathbb{R}^d)$ is defined by
	\begin{equation}
		L^{r,\rho}(\mathbb{R}^d)=: \left\{ f \text{ is measurable on } \mathbb{R}^d : \|f\|_{L^{r,\rho}}<\infty\right\}, \notag
	\end{equation}
	where
	\[
	\|f\|_{L^{r,\rho}}=: \left\{
	\begin{array}{cl}
		( \frac{\rho}{r} \int_0^\infty (s^{1/r} f^*(s))^\rho \frac{1}{s}ds)^{1/\rho} &\text{ if } \rho <\infty, \\
		\sup_{s>0} s^{1/r} f^*(s) &\text{ if } \rho=\infty.
	\end{array}
	\right.
	\]
\end{defn}

We collect the following basic properties of $L^{r,\rho}(\mathbb{R}^d)$ in the following lemmas.
\begin{lem}[Properties of Lorentz spaces \cite{ONeil}] \ 
	\begin{itemize}
		\item For $1<r<\infty$, $L^{r,r}(\mathbb{R}^d) \equiv L^r(\mathbb{R}^d)$ and by convention, $L^{\infty,\infty}(\mathbb{R}^d)= L^\infty(\mathbb{R}^d)$.
		\item For $1<r<\infty$ and $0<\rho_1<\rho_2\leq \infty$, $L^{r,\rho_1}(\mathbb{R}^d)\subset L^{r,\rho_2}(\mathbb{R}^d)$. 
		\item For $1<r<\infty$, $0<\rho \leq \infty$, and $\theta>0$, $\||f|^\theta\|_{L^{r,\rho}} = \|f\|^\theta_{L^{\theta r, \theta \rho}}$.
		\item For $b>0$, $|x|^{-b} \in L^{\frac{d}{b},\infty}(\mathbb{R}^d)$ and $\||x|^{-b}\|_{L^{\frac{d}{b},\infty}} = |B(0,1)|^{\frac{b}{d}}$, where $B(0,1)$ is the unit ball of $\mathbb{R}^d$.
	\end{itemize}
\end{lem}
\begin{lem}[H\"older's inequality \cite{ONeil}]  \ 
	\begin{itemize}
		\item  Let $1<r, r_1, r_2<\infty$ and $1\leq \rho, \rho_1, \rho_2 \leq \infty$ be such that
		\[
		\frac{1}{r}=\frac{1}{r_1}+\frac{1}{r_2}, \quad \frac{1}{\rho} \leq \frac{1}{\rho_1}+\frac{1}{\rho_2}.
		\]
		Then for any $f \in L^{r_1, \rho_1}(\mathbb{R}^d)$ and $g\in L^{r_2, \rho_2}(\mathbb{R}^d)$
		\[
		\|fg\|_{L^{r,\rho}} \lesssim \|f\|_{L^{r_1, \rho_1}} \|g\|_{L^{r_2,\rho_2}}.
		\]
		\item Let  $1<r_1,r_2<\infty $ and  $1\le \rho_1,\rho_2\le \infty $  be such that 
		\begin{equation}
			1=\frac{1}{r_1}+\frac{1}{r_2},\quad 1\le \frac{1}{\rho_1}+\frac{1}{\rho_2}.\notag
		\end{equation}
		Then  for any $f \in L^{r_1, \rho_1}(\mathbb{R}^d)$ and $g\in L^{r_2, \rho_2}(\mathbb{R}^d)$
		\begin{equation}
			\|fg\|_{L^1}\lesssim  \|f\|_{L^{r_1, \rho_1}} \|g\|_{L^{r_2,\rho_2}}.\notag
		\end{equation}
	\end{itemize}
\end{lem}
\begin{lem}[{Interpolation in Lorentz spaces} \cite{Dao}]\label{CZ}
	Let  $1<p_0<p<p_1<\infty $,  $r>0$ and $0<\theta<1$ satisfy  $\frac{1}{p}=\frac{1-\theta}{p_0}+\frac{\theta}{p_1}$. For every  $f\in L^{p_0,\infty }(\mathbb{R}^d)\cap L^{p_1,\infty }(\mathbb{R}^d)$, we have
	\begin{equation}
		\|f\|_{L^{p,r}(\mathbb{R} ^d)}\lesssim  \|f\|_{L^{p_0,\infty }(\mathbb{R} ^d)}^{1-\theta} \|f\|_{L^{p_1,\infty }(\mathbb{R} ^d)}^{\theta}.\notag
	\end{equation}
\end{lem}
\begin{lem}[Gagliardo-Nirenberg's  inequality  in Lorentz space]
	\label{L:GN}
	Let  $1 < q, r \le \infty$, $1 \le q_1, r_1 \le \infty$ and  $0 < \sigma < s < \infty$ be  such that 
	\begin{equation}
			\frac{1}{p} = \left( 1 - \frac{\sigma}{s} \right) \frac{1}{q} + \frac{\sigma}{sr}, \quad \frac{1}{p_1} = \left( 1 - \frac{\sigma}{s} \right) \frac{1}{q_1} + \frac{\sigma}{sr_1}.\notag
	\end{equation}
	Then  
\begin{equation}
	\| |\nabla |^\sigma f \|_{L^{p,p_1}(\mathbb{R}^d)} \lesssim  \| f\|_{L^{q,q_1}(\mathbb{R}^d)} ^{1 - \frac{\sigma}{s}} \||\nabla |^s f \|_{L^{r,r_1}(\mathbb{R}^d)}^{\frac{\sigma}{s}}.\notag
\end{equation}
\end{lem}
\begin{proof}
See \cite[Proposition 3.2]{Wei2023}.	
\end{proof}

\begin{lem}[Convolution inequality \cite{ONeil}] ~
	Let $1<r,r_1, r_2<\infty$ and $1\leq \rho, \rho_1, \rho_2 \leq \infty$ be such that
	\[
	1+\frac{1}{r} =\frac{1}{r_1}+\frac{1}{r_2}, \quad \frac{1}{\rho}\leq \frac{1}{\rho_1}+\frac{1}{\rho_2}.
	\]
	Then 	for any $f\in L^{r_1, \rho_1}(\mathbb{R}^d)$ and $g\in L^{r_2, \rho_2}(\mathbb{R}^d)$ 
	\[
	\|f\ast g\|_{L^{r,\rho}(\mathbb{R} ^d)} \lesssim \|f\|_{L^{r_1,\rho_1}(\mathbb{R} ^d)} \|g\|_{L^{r_2,\rho_2}(\mathbb{R} ^d)}. 
	\]
\end{lem}
Using the above convolution inequality,  one can easily obtain    Bernstein's inequality in Lorentz spaces.
\begin{lem}[Bernstein's inequality]\label{L:Bernstein}
	Let  $N>0$,  $1<r_1<r_2<\infty $ and  $1\le \rho_1\le\rho_2\le \infty $. Then 
	\begin{equation}
		 \|P_{N}f\|_{L^{r_2,\rho_2}(\mathbb{R} ^d)}\lesssim  N^{d(\frac{1}{r_1}-\frac{1}{r_2})} \|f\|_{L^{r_1,\rho_1}(\mathbb{R} ^d)}.\notag  
	\end{equation}
\end{lem}
Next,  in Lemma \ref{L:sobolev}--Lemma \ref{L:6141},  we  recall the Sobolev embedding, product rule, and chain rule in Lorentz spaces. We start by introducing the following definition.
\begin{defn}
	Let  $s\ge0,1<r<\infty $ and  $1\le\rho\le \infty $. We define the Sobolev-Lorentz spaces 
	\begin{equation}
		W^sL^{r,\rho}(\mathbb{R} ^d)=:\left\{f\in \mathcal{S}'(\mathbb{R} ^d):(1-\Delta )^{s/2}f\in L^{r,\rho}(\mathbb{R} ^d) \right\},\notag
	\end{equation}  
	\begin{equation}
		\dot W^sL^{r,\rho}(\mathbb{R} ^d)=:\left\{f\in \mathcal{S}'(\mathbb{R} ^d):(-\Delta )^{s/2}f\in L^{r,\rho}(\mathbb{R} ^d) \right\},\notag
	\end{equation}
	where  $\mathcal{S}'(\mathbb{R} ^d)$ is the space of tempered distributions on  $\mathbb{R} ^d$ and 
	\begin{equation}
		(1-\Delta )^{s/2}f=\mathcal{F}^{-1}\left((1+|\xi|^2)^{s/2}\mathcal{F} (f)\right),\qquad (-\Delta )^{s/2}f=\mathcal{F} ^{-1}(|\xi|^s\mathcal{F} (f))\notag
	\end{equation}  
	with  $\mathcal{F} $ and  $\mathcal{F} ^{-1}$ the Fourier and its inverse Fourier transforms respectively. The spaces  $W^sL^{r,\rho}(\mathbb{R} ^d)$ and  $\dot W^sL^{r,\rho}(\mathbb{R} ^d)$ are endowed respectively with the norms
	\begin{equation}
		\|f\|_{W^sL^{r,\rho}} = \|f\|_{L^{r,\rho}}+ \|(-\Delta )^{s/2}f\|_{L^{r,\rho}},\qquad  \|f\|_{\dot WL^{r,\rho}} = \|(-\Delta )^{s/2}f\|_{L^{r,\rho}}.\notag   
	\end{equation}    
	For simplicity, when  $s=1$ we write  $WL^{r,\rho}:=W^1L^{r,\rho}(\mathbb{R} ^d)$  and  $\dot WL^{r,\rho}:=\dot W^1L^{r,\rho}(\mathbb{R} ^d)$.  
	\end{defn}

\begin{lem}[Sobolev embedding \cite{DinhKe}]\label{L:sobolev}
	Let  $1<r<\infty ,1\le \rho\le \infty $ and  $0<s<\frac{d}{r}$. Then 
	\begin{equation}
		\|f\|_{L^{\frac{dr}{d-sr},\rho}(\mathbb{R} ^d)}\lesssim   \|(-\Delta )^{s/2}f\|_{L^{r,\rho}(\mathbb{R} ^d)} \quad\text{for any}\quad   f\in \dot W^sL^{r,\rho}(\mathbb{R} ^d).\notag
	\end{equation}  
\end{lem}
\begin{lem}[Product rule \cite{Cruz}]\label{L:leibnitz}
	Let  $s\ge0,1<r,r_1,r_2,r_3,r_4,<\infty $, and   $1\le\rho,\rho_1,\rho_2,\rho_3,\rho_4\le\infty $ be such that 
	\begin{equation}
		\frac{1}{r}=\frac{1}{r_1}+\frac{1}{r_2}=\frac{1}{r_3}+\frac{1}{r_4},\qquad \frac{1}{\rho}=\frac{1}{\rho_1}+\frac{1}{\rho_2}=\frac{1}{\rho_3}+\frac{1}{\rho_4}.\notag
	\end{equation} 
	Then for any  $f\in \dot W^sL^{r_1,\rho_1}(\mathbb{R} ^d)\cap L^{r_3,\rho_3}(\mathbb{R} ^d)$  and  $g\in \dot W^sL^{r_4,\rho_4}(\mathbb{R} ^d)\cap L^{r_2,\rho_2}(\mathbb{R} ^d)$,  
	\begin{equation}
		\|(-\Delta )^{s/2}(fg)\|_{L^{r,\rho}} \lesssim   \|(-\Delta )^{s/2}f\|_{L^{r_1,\rho_1}} \|g\|_{L^{r_2,\rho_2}}+ \|f\|_{L^{r_3,\rho_3}} \|(-\Delta )^{s/2}g\|_{L^{r_4,\rho_4}}.\notag    
	\end{equation}
\end{lem}

\begin{lem}[Chain rule\cite{Aloui}]\label{L:6141}
	Let  $s\in [0,1],F\in C^1(\mathbb{C},\mathbb{C})$ and  $1<p,p_1,p_2<\infty $,  $1\le q,q_1,q_2<\infty $ be  such that 
	\begin{equation}
		\frac{1}{p}=\frac{1}{p_1}+\frac{1}{p_2},\qquad\frac{1}{q}=\frac{1}{q_1}+\frac{1}{q_2}.\notag
	\end{equation}   
	Then
	\begin{equation}
		 \|(-\Delta )^{s/2}F(f)\|_{L^{p,q}(\mathbb{R} ^d)}\lesssim  \|F'(f)\|_{L^{p_1,q_1}(\mathbb{R} ^d)} \|(-\Delta )^{s/2}f\|_{L^{p_2,q_2}(\mathbb{R} ^d)}.\notag   
	\end{equation}
\end{lem}
Next, we establish the Coifman-Meyer lemma in Lorentz spaces, which will be used to prove Lemma  \ref{L:error}  in the appendix.
\begin{lem}
	\label{L:CF}
	Let  $m(\xi,\eta)$ be a Coifman-Meyer multiplier on  $\mathbb{R} ^d\times \mathbb{R} ^d$, i.e.  $m(\xi,\eta)$ satisfies 
	\begin{equation}
		|\partial_{\xi}^\alpha \partial_{\eta}^\beta m(\xi,\eta)|\lesssim (1+|\xi|+|\eta|)^{-|\alpha |-|\beta|}\quad\text{for any}\quad \alpha ,\beta\in \mathbb{N} ^d.\notag
	\end{equation}
	Let 
	\begin{equation}
		T_m(f,g)(x)=:\int _{\mathbb{R} ^d}\int _{\mathbb{R} ^d}e^{ix\cdot (\xi+\eta)}m(\xi,\eta)\hat f(\xi)\hat g(\eta)d\xi d\eta .\notag
	\end{equation}
	Then for any  $1<p,p_0,p_1<\infty $ with  $\frac{1}{p}=\frac{1}{p_0}+\frac{1}{p_1}$,
	\begin{equation}
		 \|T_m(f,g)\|_{L^{p,2}(\mathbb{R} ^d)}\lesssim   \|f\|_{L^{p_0,2}(\mathbb{R} ^d)} \|g\|_{L^{p_1,\infty }(\mathbb{R} ^d)}.\label{6122}
	\end{equation}
\end{lem}
\begin{proof}
	We first claim that for any  $1<r,r_0,r_1<\infty $ with  $\frac{1}{r}=\frac{1}{r_0}+\frac{1}{r_1}$, the following estimate holds: 
	\begin{equation}
		\label{C:6271}
				 \|T_m(f,g)\|_{L^{r,\infty }(\mathbb{R} ^d)}\lesssim   \|f\|_{L^{r_0}(\mathbb{R} ^d)} \|g\|_{L^{r_1,\infty }(\mathbb{R} ^d)}.
	\end{equation}
	In fact, fixing  $\varepsilon _0>0$ sufficiently small and letting 
	\begin{equation}
		\frac{1}{r^{\pm}}=:\frac{1}{r}\pm \varepsilon _0,\qquad \frac{1}{r_1 ^{\pm}}=:\frac{1}{r_1}\pm \varepsilon _0\quad\text{ such that }\quad  \frac{1}{r^{\pm}}=\frac{1}{r_0}+\frac{1}{r_1 ^{\pm}},\notag
	\end{equation}
	we deduce from Coifman-Meyer multiplier  theorem that 
	\begin{equation}
	\|T_m(f,g)\|_{L^{r^{\pm}}(\mathbb{R} ^d)}\lesssim   \|f\|_{L^{r_0}(\mathbb{R} ^d)}  \|g\|_{L^{r_1^\pm   }(\mathbb{R} ^d)}.\notag  
\end{equation}
Hence by  Hunt's interpolation theorem \cite[Theorem IX.19]{ReedSimon1975}, 
\begin{equation}
	T_m(f,\cdot):\ (L^{r_1^+}(\mathbb{R} ^d),L^{r_1^-}(\mathbb{R} ^d))_{\frac{1}{2},\infty }\rightarrow(L^{r^+}(\mathbb{R} ^d),L^{r^-}(\mathbb{R} ^d))_{\frac{1}{2},\infty }\notag
\end{equation}
is bounded with bound  $ \|f\|_{L^{r_0}(\mathbb{R} ^d)} $.
Further utilizing the real interpolation of Lebesgue spaces (\cite[Theorem 5.2.1]{Bergh1976}), we obtain (\ref{C:6271}).  
	
	We now prove (\ref{6122}).  Fix  $\varepsilon _0>0$  sufficiently small and let 
	\begin{equation}
		\frac{1}{p^{\pm}}=:\frac{1}{p}\pm \varepsilon _0,\qquad \frac{1}{p_0 ^{\pm}}=:\frac{1}{p_0}\pm \varepsilon _0\quad\text{ such that }\quad  \frac{1}{p^{\pm}}=\frac{1}{p_0^{\pm}}+\frac{1}{p_1  }.\notag
	\end{equation}
 By Claim (\ref{C:6271}), 
	\begin{equation}
				 \|T_m(f,g)\|_{L^{p^{\pm},\infty }(\mathbb{R} ^d)}\lesssim  \|f\|_{L^{p^{\pm}_0}(\mathbb{R} ^d)} \|g\|_{L^{p_1,\infty }(\mathbb{R} ^d)}.\notag   
	\end{equation}
Again by   Hunt's interpolation theorem \cite[Theorem IX.19]{ReedSimon1975}, 
\begin{equation}
	T_m(\cdot,g):\ \ (L^{p_0^+}(\mathbb{R} ^d),L^{p_0^-}(\mathbb{R} ^d))_{\frac{1}{2},2 }\rightarrow(L^{p^+,\infty }(\mathbb{R} ^d),L^{p^-,\infty }(\mathbb{R} ^d))_{\frac{1}{2},2 }\notag
\end{equation}
is bounded with bound  $ \|g\|_{L^{p_1,\infty }(\mathbb{R} ^d)} $. 
	By real interpolation of Lorentz spaces (\cite[Theorem 5.3.1]{Bergh1976}), we obtain  (\ref{6122}).
\end{proof}

At the end of this Subsection, we record two nonlinear estimates that will be used in the following sections.
\begin{lem}
	\label{L:nonlinear estimate}
	Let  $d\ge3$, $0<b<\min \left\{2,\frac{d}{2}\right\} $ and $0<s<\min \left\{ \frac{d-2b}{2}+1,\frac{4-2b}{d}+1 \right\}$. Then on any spacetime slab  $I\times \mathbb{R} ^d$, we have 
	\begin{equation}
		\||\nabla |^{s}(|x|^{-b}|u|^{\frac{4-2b}{d}}u)\|_{L^2_tL_x^{\frac{2d}{d+2},2}}\lesssim   \|u\|_{L^\infty _tL_x^2}^{\frac{4-2b}{d}} \||\nabla |^{s}u\| _{L^2_tL_x^{\frac{2d}{d-2},2}} + \||\nabla |^{\frac{2s}{\gamma }}u\|^{\frac{\gamma }{2}}_{L_t^\gamma L_x^{\rho,2}},\notag 
	\end{equation}
	and 
	\begin{equation}
		\||\nabla |^{s}(|x|^{-b}|u|^{\frac{4-2b}{d}}u)\|_{L^2_tL_x^{\frac{2d}{d+2},2}}\lesssim   \|u\|_{L^\gamma  _tL_x^{\rho,2}}^{\frac{4-2b}{d}} \||\nabla |^{s}u\| _{L^\gamma _tL_x^{\rho,2}} ,\notag
	\end{equation}
where 	$(\gamma ,\rho)$ was defined by (\ref{E:gamma}). 
\end{lem}
\begin{proof}
	By Lemma \ref{L:leibnitz},  
	\begin{align}
		\||\nabla |^{s}(|x|^{-b}|u|^{\frac{4-2b}{d}}u)\|_{L^2_tL_x^{\frac{2d}{d+2},2}} 
		&\lesssim  \||x|^{-b}\|_{L_x^{\frac{d}{b},\infty }} \||\nabla |^s (|u|^{\frac{4-2b}{d}}u)\|  _{L^2_tL_x^{\frac{2d}{d+2-2b},2}} \notag\\
		&+
		\||x|^{-(b+s)}\|_{L_x^{\frac{d}{b+s},\infty }} \||u|^{\frac{4-2b}{d}}u\|_{L^2_tL_x^{\frac{2d}{d+2-2(b+s)},2}} .\notag 
	\end{align}
Noting that 
	\begin{equation}
		\frac{d+2-2b}{2d}=\frac{4-2b}{d}\cdot\frac{1}{2}+\frac{d-2}{2d}=\frac{1}{\rho}\cdot(\frac{4-2b}{d}+1),\   \frac{1}{2}=\frac{1}{\gamma }\cdot (\frac{4-2b}{d}+1). \label{zb}
	\end{equation}
	it follows from Lemma \ref{L:6141} that 
	\begin{align}
		&\||\nabla |^s(|u|^{\frac{4-2b}{d}}u)\|_{L^2_tL_x^{\frac{2d}{d+2-2b},2}} \notag\\
		&\lesssim \min \left\{  \|u\|_{L^\infty _tL_x^{ 2}}^{\frac{4-2b}{d}} \||\nabla |^{s}u\| _{L^2_tL_x^{\frac{2d}{d-2},2}}, \|u\|^{\frac{4-2b}{d}}_{L^\gamma _tL_x^{\rho,2}} \||\nabla |^{s}u\|_{L^\gamma _tL_x^{\rho,2}}   \right\}.\notag
	\end{align}
	On the other hand, since $s<\frac{d-2b}{2}+1$, we have that  $\frac{1}{\rho}>\frac{2s}{d\gamma }$ and 
	\begin{equation}
		\frac{d+2-2(b+s)}{2d}= (\frac{4-2b}{d}+1)\cdot (\frac{1}{\rho}-\frac{2s}{d\gamma })=\frac{\gamma }{2}\cdot (\frac{1}{\rho}-\frac{2s}{d\gamma }).\notag
	\end{equation}
	By H\"older and Sobolev's embedding in Lorentz space (Lemma \ref{L:sobolev}), 
	\begin{equation}
		\||u|^{\frac{4-2b}{d}}u\|_{L^2_tL_x^{\frac{2d}{d+2-2(b+s)},2}} \lesssim   \||\nabla |^{\frac{2s}{\gamma }}u\|^{\frac{\gamma }{2}}_{L_t^\gamma L_x^{\rho,2}}\lesssim \|u\| _{L_t^\gamma L_x^{\rho,2}}^{\frac{\gamma }{2}-1} \||\nabla |^{s}u\|_{L_t^\gamma L_x^{\rho,2}},\notag
	\end{equation}
	where we used the Gagliardo-Nirenberg's inequality (Lemma \ref{L:GN}) in the second inequality.  
	Combining the above estimates, we obtain the desired estimate in Lemma \ref{L:nonlinear estimate}.  
\end{proof}
\begin{lem}[\cite{Visan2007}]\label{L:Visan}
	Let $H$ be a Hölder continuous function of order $0 <\alpha  < 1$. Then, for every $0 < \sigma < \alpha $, $1 < p < \infty$, and $\sigma /  \alpha < s < 1$, we have
	\[
	\| |\nabla|^\sigma H(u) \|_{L^p(\mathbb{R} ^d)} \lesssim \| u \|_{L^{p_1}(\mathbb{R} ^d)}^{\alpha - \sigma / s} \|| \nabla|^s u \| ^{\sigma /s}_{L^{\sigma /sp_1}(\mathbb{R} ^d)}
	\]
	provided $\frac{1}{p} = \frac{1}{p_1} + \frac{1}{p_2}$ and $(1 - \sigma / ( \alpha  s))p_1 > 1$.
\end{lem}

\subsection{Strichartz estimate, local well-posedness and stability.}
In this Subsection, we recall Strichartz estimate in the Lorentz space, the  local well-posedness and stability results of the Cauchy probelm \ref{nls}.  
\begin{defn}[Admissibility] 
	A pair $(q,r)$ is said to be Schr\"odinger admissible, for short $(q,r)\in \Lambda $, where
	\begin{equation}
		\Lambda=\left\{(q,r): 2\le q,r\le \infty,\  	\frac{2}{q}+\frac{d}{r}=\frac{d}{2},\ (q,r,d)\neq (2,\infty ,2)\right\}.\notag
	\end{equation}
\end{defn}

The following result is a key tool for our work. It is
established in \cite[Theorem 10.1]{Keel-Tao}.

\begin{prop}[Strichartz estimates]\label{P:SZ}\ \\
	\begin{itemize}
		\item Let $(q,r)\in \Lambda $ with $r<\infty$. Then for any $f\in L^2(\mathbb{R}^d)$
		\begin{align}  
			\|e^{it\Delta }f\|_{L^q_t L^{r,2}_x(\mathbb{R}\times \mathbb{R}^d)} \lesssim \|f\|_{L^2_x(\mathbb{R}^d)}.\notag
		\end{align}
		
		\item Let $(q_1, r_1), (q_2,r_2)\in \Lambda $ with $r_1, r_2<\infty$, $t_0\in \mathbb{R}$ and $I\subset \mathbb{R}$ be an interval containing $t_0$. Then 		for any $F\in L_t^{q_2'}L^{r_2',2}_x(I\times \mathbb{R}^d)$ 
		\begin{align} 
			\left\|\int_{t_0}^t e^{i(t-\tau)\Delta } F(\tau) d\tau\right\|_{L^{q_1}_tL^{r_1,2}_x(I\times \mathbb{R}^d)} \lesssim \|F\|_{L^{q_2'}_tL^{r_2',2}_x(I\times \mathbb{R}^d)}.\notag
		\end{align}
	\end{itemize}
\end{prop}
\begin{prop}[Local well-posedness]\label{T:CP}
	For any  $u_0\in L^2(\mathbb{R} ^d)$ and  $t_0\in \mathbb{R} $, there exists  a unique  maximal solution  $u:(-T_-(u_0),T_+(u_0))\times \mathbb{R} ^d\rightarrow \mathbb{C}$  to (\ref{nls}) with  $u(t_0)=u_0$.     This solution conserves its mass and  also has the following properties: \\
	(a) If  $T_+=T_+(u_0)<\infty $, then  $ \|u\|_{L^\gamma _tL_x^{\rho,2}((0,T_+)\times \mathbb{R} ^d)} =+\infty $. An analogous result holds for  $T_-(u_0)$.  \\
	(b) If  $\|u\|_{L^\gamma _tL_x^{\rho,2}((0,T_+)\times \mathbb{R} ^d)}<+\infty  $, then  $T_+=\infty $ and  $u$ scatters as  $t\rightarrow+\infty $.      An analogous result holds for  $T_-(u_0)$. \\
	(c) If $u_0^{(n)}$ is a sequence converging to $u_0$ in $L_x^2(\mathbb{R}^d)$ and $u^{(n)} : I^{(n)} \times \mathbb{R}^d \to \mathbb{C}$ are the associated maximal-lifespan solutions, then $u^{(n)}$ converges locally uniformly to $u$.\\
	(d)  For any $\psi\in L^2(\mathbb{R}^d)$, there exist $T>0$ and a unique solution $u:(T,\infty)\times\mathbb{R}^d\to\mathbb{C}$ to \eqref{nls} obeying $e^{-it\Delta}u(t)\to \psi$ in $L^2$ as $t\to\infty$.  The analogous statement holds backward in time.\\
	(e)    If   $M(u_0)$   is sufficiently small  depending on  $d$,  then 	 $u$
	is a global solution and scatter.  Moreover, 
	\begin{equation}
		\|u\|_{L^{\gamma }_t L_x^{\rho,2}(\mathbb{R} \times \mathbb{R} ^d)}\lesssim M(u).\notag 
	\end{equation}
\end{prop}
\begin{proof}
	By considering Strichartz estimates in Lorentz spaces, Aloui and Tayachi \cite{Aloui} proved   (a) and (c).  The proofs of (b), (d), and (e) are based on Lemma \ref{L:nonlinear estimate} and the standard arguments outlined in \cite[Chapter 7]{Cazenave}.
\end{proof}
\begin{prop} [Stability]\label{P:stab} Suppose $\tilde u:I\times\mathbb{R}^d\to\mathbb{C}$ obeys
	\begin{equation}
		\|\tilde u\|_{L_t^\infty  L_x^2(I\times\mathbb{R}^d)} + \|\tilde u\| _{L_t^\gamma L_x^{\rho,2}(I\times \mathbb{R} ^d)}\leq E<\infty. \notag
	\end{equation}
	There exists $\varepsilon _1 = \varepsilon_1(E)>0$ such that if
	\begin{align*}
		\|  (i\partial_t+\Delta)\tilde u -\mu  |x|^{-b}|\tilde u|^\frac{4-2b}{d} \tilde u \|_{L_t^2 L_x^{\frac{2d}{d+2},2}(I\times \mathbb{R} ^d)}& \leq \varepsilon<\varepsilon_1,\\
		\|e^{i(t-t_0)\Delta}[u_0-\tilde u|_{t=t_0}]\|_{L_t^\gamma L_x^{\rho,2}(I\times \mathbb{R} ^d)}& \leq \varepsilon<\varepsilon_1,
	\end{align*}
	for some $t_0\in I$ and $u_0\in L^2(\mathbb{R}^d)$ with $\|u_0-\tilde u|_{t={t_0}}\|_{L^2}\lesssim_E 1$, then there exists a unique solution $u:I\times\mathbb{R}^d\to\mathbb{C}$ to (\ref{nls}) with $u|_{t=t_0}=u_0$, which satisfies
	\begin{equation}
		\|u-\tilde u\|_{L_t^\gamma L_x^{\rho,2}(I\times \mathbb{R} ^d)}\lesssim \varepsilon \quad\text{and}\quad \|u\|_{L_t^\infty  L_x^2(I\times\mathbb{R}^d)}+\|u\|_{L_t^\gamma L_x^{\rho,2}(I\times \mathbb{R} ^d)}\lesssim_E 1.\notag
	\end{equation}
\end{prop}
\begin{proof}
	The proof  follows  from   Lemma \ref{L:nonlinear estimate}  and the standard arguments  (see e.g. \cite{KV}),  so we omit the details. 
\end{proof}

The following  bilinear Strichartz estimates  will be used to derive the long time Strichartz estimate in Section \ref{S:5}. 
\begin{lem}[Bilinear Strichartz estimates]\label{L:bilinear SZ}
	Suppose  $\hat v(t,\xi)$ is supported on  $|\xi |\le M$ and  $\hat u(t,\xi)$ is  supported on  $|\xi |>N$,   $M\le N$. Let  $I=[a,b], d\ge1$ and  
	\begin{equation}
		\|u\|_{S^0_*(I\times \mathbb{R} ^d)}  := \|u(a)\|_{L^2}+ \|(i\partial_{t}+\Delta )u\|_{L_t^ 2 L_x^{\frac{2d}{d-2},2}(I\times \mathbb{R} ^d)}.  \label{523x1}
	\end{equation}
	Then 
	\begin{equation}
		\|uv\|_{L^2_{t,x}(I\times \mathbb{R} ^d)}\lesssim  \frac{M^{(d-1)/2}}{N^{1/2}}  \|u\|_{S^0_*(I\times \mathbb{R} ^d)}   \|v\|_{S^0_*(I\times \mathbb{R} ^d)} , \label{E:bilinear SZ}
	\end{equation}        
	where the implicit constant independent of  $I$.  
\end{lem}
\begin{proof}
	The proof   is almost identical to that of \cite[Lemma 2.5]{Visan2007}, with the only difference being that we need to replace the  Strichartz   estimates with those estimates in Lorentz spaces.  
\end{proof}
\begin{lem}[Regularity persistence]
	\label{L:S_0*}
	Let  $u(t,x)$ be a  solution to  (\ref{nls})  such that   
	\begin{equation}
		\|u\|_{L^\gamma _tL_x^{\rho,2}(I\times \mathbb{R} ^d)}\lesssim1,  \notag
	\end{equation}
where  $(\gamma ,\rho)$ was defined in (\ref{E:gamma}).  	 Then  
	\begin{equation}
		\|u\| _{S^0_*(I\times \mathbb{R} ^d)} \lesssim 1.\label{614ww1}
	\end{equation}
	More generally, suppose  $\frac{1}{2}\le s<\min \left\{ \frac{d-2b}{2}+1,\frac{4-2b}{d}+1 \right\}$, then 
	\begin{equation}
		\||\nabla |^s u\|_{S^0_*(I\times \mathbb{R} ^d)}\lesssim   \|u\|_{L^\infty _t\dot H^s_x(I\times \mathbb{R} ^d)}. \label{614ww2}
	\end{equation}
\end{lem}
\begin{proof}
  (\ref{614ww1}) follows immediately from  H\"older's inequality and  (\ref{zb}): 
	\begin{equation}
		\|u\| _{S^0_*(I\times \mathbb{R} ^d)} 	\lesssim   \|u_0\|_{L^2(\mathbb{R} ^d)}+ \||x|^{-b}\|_{L^{\frac{d}{b},\infty }(\mathbb{R} ^d)} \|u\|_{L^\gamma _tL^{\rho,2}_x (I\times \mathbb{R} ^d)}^{\frac{4-2b}{d} +1}\lesssim 1. \notag
	\end{equation}
	To prove (\ref{614ww2}), we use Duhamel's formula and Lemma \ref{L:nonlinear estimate} to obtain 
	\begin{align}
		\||\nabla |^s u\| _{L^\gamma _tL_x^{\rho,2}(I\times \mathbb{R} ^d)}&\lesssim   \|u_0\|_{\dot H^s(\mathbb{R} ^d)}+ \||\nabla |^s(|x|^{-b}|u|^{\frac{4-2b}{d} }u)\|_{L^2_tL_x^{\frac{2d}{d+2},2}(I\times \mathbb{R} ^d)}\notag\\
		&\lesssim   \|u_0\|_{\dot H^s(\mathbb{R} ^d)}+ \|u\|^\frac{4-2b}{d} _{L^\gamma _t L_x^{\rho,2}(I\times \mathbb{R} ^d)}  \||\nabla |^s u\|_{L^\gamma _tL_x^{\rho,2}(I\times \mathbb{R} ^d)}.\notag     
	\end{align}
	By the continuity method, 
	\begin{equation}
		\||\nabla |^s u\| _{L^\gamma _tL_x^{\rho,2}(I\times \mathbb{R} ^d)}\lesssim   \|u\|_{L_t^\infty \dot H^s_x(I\times \mathbb{R} ^d)}.\notag 
	\end{equation}
	Therefore, by Duhamel's formula and Strichartz estimate, 
	\begin{align}
		&	\||\nabla |^s u\|_{S^0_*(I\times \mathbb{R} ^d)} 
		\lesssim  \|u_0\|_{\dot H^s(\mathbb{R} ^d)}+ \||\nabla |^s(|x|^{-b}|u|^{\frac{4-2b}{d} }u)\|_{L^2_tL_x^{\frac{2d}{d+2},2}(I\times \mathbb{R} ^d)}\notag\\
		&\lesssim   \|u_0\|_{\dot H^s(\mathbb{R} ^d)}+ \|u\|^\frac{4-2b}{d} _{L^\gamma _t L_x^{\rho,2}(I\times \mathbb{R} ^d)}   \||\nabla |^s u\|_{L^\gamma _tL_x^{\rho,2}(I\times \mathbb{R} ^d)}\lesssim   \|u\|_{L^\infty _t\dot H^s_x(I\times \mathbb{R} ^d)}.\notag 
	\end{align}
\end{proof}
\subsection{The construction of  scattering solutions away from the origin.}\label{s:2.4}
The following Proposition establishes the existence of scattering solutions to \eqref{nls} associated with initial data living sufficiently far from the origin.  As a benefit of this, the spatial center and the frequency center  for the almost periodic solutions in Section \ref{S:3} can be chosen as  $x(t)\equiv\xi (t)\equiv0$.

\begin{prop}\label{P:embed} Let   $x_n\in \mathbb{R}^d$ and $\xi_n\in\mathbb{R}^d$ satisfy
 \begin{equation}
 	\lim_{n\rightarrow \infty }\left(|x_n|+|\xi_n|\right)=\infty .\notag
 \end{equation}
	Let $\phi\in L^2(\mathbb{R}^d)$ and define
	\begin{equation}
		\phi_n(x) =: e^{ix\cdot \xi_n}\phi (x-x_n). \notag
	\end{equation}
	Then there exists a subsequence $n_k$ and a global scattering solution $v_{n_k}$ to \eqref{nls} satisfying
	\[
	v_{n_k}(0) = \phi_{n_k}  \quad\text{and}\quad \| v_{n_k}\|_{L^\gamma_tL^{\rho,2}_x (\mathbb{R}\times \mathbb{R}^d)} \lesssim  \|\phi\|_{L^2},
	\] 	
		where the implicit constant independent of  $k$.  
\end{prop}
\begin{proof}
	The proof is inspired by \cite{MMZ}. Let 
	\begin{equation}
		v_n(t,x)=:e^{it\Delta }\left(e^{ix\cdot \xi_n}\phi (x-x_n)\right)=e^{-it|\xi_n|^2}e^{ix\cdot \xi_n}e^{it\Delta }\phi (x-2t\xi_n-x_n),\notag
	\end{equation}
	and 
		\begin{equation}
		e_{n}(t,x)=:  -\mu |x|^{-b}|  v_{n}(t,x)|^\frac{4-2b}{d}   v_{n}(t,x). \notag
	\end{equation}
	Assume that we can find a subsequence  $n_{k}$  such that  
	\begin{equation}
		\lim_{k\rightarrow\infty } \|e_{n_k}(t,x)\| _{L^2_tL_x^{\frac{2d}{d+2},2}(\mathbb{R} \times \mathbb{R} ^d)}=0.\label{Limit}
	\end{equation}
	Then  Proposition \ref{P:embed} follows immediately from (\ref{Limit}) and the  stability results in Proposition \ref{P:stab}.  
	
	Therefore, it sufficies to prove (\ref{Limit}).   For any \(\varepsilon >0\), we take a \(\psi (t,x)\in C_c^\infty (\mathbb{R} \times \mathbb{R} ^d)\) such that
	\begin{equation}
		\|\psi (t,x)-e^{it\Delta }\phi(x)\|_{L^ \gamma  _tL_x^{\rho,2}(\mathbb{R} \times \mathbb{R} ^d)} <\varepsilon .
	\end{equation}
	
	Fix \(T>0\).  	We  first consider the   estimate  of \(e_{n}\) on \([-T,T]\).  
	Let 
	\begin{equation}
		\Phi_n(t,x)=:-\mu |x|^{-b}|  \psi(t,x-2t\xi_n-x_n)|^\frac{4-2b}{d}   \psi (t,x-2t\xi_n-x_n).\notag
	\end{equation}
	Then by H\"older's inequality, we have, on  $[-T,T]\times \mathbb{R} ^d$,  
\begin{align}
		&\|e_{n}(t,x) -\Phi_n(t,x)\| _{L_t^2L_x^{\frac{2d}{d+2},2}([-T,T]\times \mathbb{R} ^d)} \notag\\
		&\lesssim  \left( \|e^{it\Delta }\phi\|_{L_t^\gamma L_x^{\rho,2}}^{\frac{4-2b}{d}}+ \|\psi (t,x)\|_{L_t^\gamma L_x^{\rho,2}}^{\frac{4-2b}{d}}  \right) \|\psi (t,x)-e^{it\Delta }\phi (x)\|_{L_t^\gamma L_x^{\rho,2}}    \lesssim   \varepsilon .\label{716w7}
\end{align}
Next, we claim that there exists a subsequence  $n_k$  such that 
\begin{equation}
	\limsup _{T\rightarrow \infty }\limsup _{k\rightarrow\infty }  \|\Phi_{n_k}\|_{L^2_t L_x^{\frac{2d}{d+2},2}([-T,T]\times \mathbb{R} ^d)}=0.\label{C716} 
\end{equation}
	To  prove (\ref{C716}), 	we consider two cases. \\
	\noindent\textbf{Case 1:} \(\frac{|\xi_n|}{|x_n|}+\frac{|x_n|}{|\xi_n|}\) is bounded. Without loss of generality, we assume there exists \(n_k\) and \(\alpha _0\in (0,\infty )\) such that
	\begin{equation}
		\lim_{k\rightarrow\infty }\frac{|x_ {n_k}|}{|\xi_{n_k}|}=\alpha _0.
	\end{equation}
	In particular, we have  $\lim_{k\rightarrow\infty }|x_{n_k}|=|\xi_ {n_k}|=\infty $.   
	When \(|t|<\frac{\alpha _0}{2}-\frac{1}{T}\) and  $x$ is bounded, we have, for  $k$  sufficiently large, 
	\begin{equation}
		|x+2t\xi_{n_k}+x_{n_k}|\ge |x_{n_k}|-2|t||\xi _{n_k}|-|x|\ge \frac{1}{T}|\xi_ {n_k}|-|x|\ge  \frac{1}{2T}|\xi_ {n_k}|.\notag
	\end{equation}
	Similarly, when \(|t|>\frac{\alpha _0}{2}+\frac{1}{T}\) and  $x$ is bounded, we have, for  $k$  sufficiently large, 
	\begin{equation}
		|x+2t\xi_{n_k}+x_{n_k}|\ge 2|t||\xi _{n_k}|- |x_{n_k}|-|x|\ge \frac{1}{T}|\xi_ {n_k}|-|x|\ge  \frac{1}{2T}|\xi_ {n_k}|.\notag
	\end{equation}
	Therefore, letting  $I_T=:[-T,T]\cap \left\{t:|t|<\frac{\alpha _0}{2}-\frac{1}{T} \right\}\cup \left\{t:|t|>\frac{\alpha _0}{2}+\frac{1}{T} \right\}$, we have 
	\begin{equation}
			  \|\Phi_{n_k}(t,x)\|_{L^2_tL_x^{\frac{2d}{d+2},2}(I_T\times \mathbb{R} ^d)}\lesssim  T^{b}|\xi _{n_k}|^{-b} \||\psi|^{\frac{4-2b}{d}}\psi\|_{L^2_tL_x^{\frac{2d}{d+2},2}(I_T\times \mathbb{R} ^d)} .\label{716w8}
	\end{equation}
	On the other hand,  by H\"older's inequality, 
	\begin{equation}
		 \|\Phi_{n_k}(t,x)\|_{L^2_tL_x^{\frac{2d}{d+2},2}(\{\frac{\alpha _0}{2}-\frac{1}{T}<|t|<\frac{\alpha _0}{2}+\frac{1}{T}\}\times \mathbb{R} ^d)} \lesssim   \|\psi(t,x)\|_{L^\gamma _tL_x^{ \rho,2}(\{\frac{\alpha _0}{2}-\frac{1}{T}<|t|<\frac{\alpha _0}{2}+\frac{1}{T}\}\times \mathbb{R} ^d)} ^{\frac{4-2b}{d}+1}.\notag
	\end{equation}
	Combining the above estimate with (\ref{716w8}), and using the monotone convergence theorem, we obtain (\ref{C716}).

		\noindent\textbf{Case 2:} \(\frac{|\xi_n|}{|x_n|}+\frac{|x_n|}{|\xi_n|}\) is unbounded. Without loss of generality, we assume there exists \(n_k\) such that
	\begin{equation}
		\lim_{k\rightarrow\infty }\frac{|x_ {n_k}|}{|\xi_{n_k}|}=\infty .\label{716x1}
	\end{equation}
	Hence, when  $x$ is bounded, we have, for  $k$  sufficiently large, 
	\begin{equation}
		|x+2t\xi_ {n_k}+x_{n_k}|\ge |x_{n_k}|-2|T||\xi _{n_k}|-|x|\ge \frac{1}{2}|x_{n_k}|.\notag
	\end{equation}
	Proceeding as in the previous case, we obtain (\ref{C716}).  
	
	Finally, we consider the estimate of  $e_n(t,x)$ on the interval  $\left\{t:|t|>T \right\}$. By H\"older's inequality,
	\begin{equation}
		\|e_{n}\|_{L_{t}^{2}L_x^{\frac{2d}{d+2},2}(\{|t|>T\}\times \mathbb{R} ^d)} \lesssim   \|e^{it\Delta }\phi\|_{L_{t}^{\gamma}L_x^{\rho,2}(\{|t|>T\}\times \mathbb{R} ^d)}^{\frac{4-2b}{d}+1},\notag
	\end{equation}
	which tends to  $0$ as  $T\rightarrow\infty $ by Strichartz and the monotone convergence theorem. This combined with (\ref{716w7}) and (\ref{C716}) yields (\ref{Limit}), and the proof of  Proposition \ref{P:stab} is completed.  
\end{proof}

\section{Reduction to almost periodic solutions}\label{S:3}
The aim of this section is to demonstrate that if   Theorem \ref{T:1} fails, then one could construct minimal counterexamples. As a consequence of minimality,  the  minimal counterexamples will satisfy   the  properties listed in Subsection \ref{S:3.2}.
 
 \subsection{Reduction to almost periodic solutions}\label{S:3.1}
   The main ingredient we need to prove Theorem \ref{T:reduction}  is the following linear profile decomposition of \cite[Theorem 5.4]{BV2007}.  
\begin{prop}[Linear profile decomposition]\label{P:LPD}
	Let $u_n$ be a bounded sequence in $L^2(\mathbb{R}^d)$.  Then after passing to a subsequence if necessary,		
	there exist $J^*\in\mathbb{N}\cup\{\infty\}$; profiles $\phi^j\in  L^2 \setminus\{0\}$; scales $\lambda_n^j\in(0,\infty)$; space translation parameters $x_n^j\in\mathbb{R}^d$;  frequency translation parameters $\xi_n^j\in\mathbb{R}^d$; time translation parameters $t_n^j$; and remainders $w_n^J$ such that the following  decomposition holds for $1\leq J\leq J^*$:
\begin{equation}
		u_n = \sum_{j=1}^J (\lambda ^j_n)^{-\frac{d}{2}}e^{ix\cdot \xi^j_n}[e^{it_n^j\Delta}\phi^j] (\frac{x-x_n^j}{\lambda _n^j}) + w_n^J;\label{21w3}
\end{equation}
	here  $w_n^J\in L^2_x(\mathbb{R}^d)$ obey
	\begin{equation}\label{vanishing0}
		\limsup_{J\to J^*}\limsup_{n\to\infty} \|e^{it\Delta}w_n^J\|_{L_{t,x}^{\frac{2(d+2)}{d}} (\mathbb{R}\times\mathbb{R}^d)} = 0.
	\end{equation}
	Moreover, for any  $j\neq k$,
	\begin{equation}\label{orthogonality}
		\lim_{n\to\infty} \biggl\{\log\bigl[\tfrac{\lambda_n^j}{\lambda_n^k}\bigr] + \tfrac{|x_n^j-x_n^k|^2}{\lambda_n^j\lambda_n^k} +\lambda_n^j\lambda_n^k|\xi_n^j-\xi_n^k|^2+ \tfrac{|t_n^j(\lambda_n^j)^2-t_n^k(\lambda_n^k)^2|}{\lambda_n^j\lambda_n^k}\biggr\} = \infty.
	\end{equation}
	In addition, we may assume that  either  $ t_n^j\equiv 0$ or  $t_n^j\to\pm\infty$, either  $  x_n^j\equiv 0$ or  $\frac{|x_n^j|}{\lambda _n^j}\to\infty$, and that either  $\xi_n^j\equiv0$ or  $\lambda_n^j\xi_n^j\to \infty $.   
		Furthermore, for any  $J\ge1$ we have the  mass decoupling property
	\begin{equation}\label{mass-decoupling}
		\lim_{n\rightarrow\infty}  [M(u_n)-\sum_{j=1}^{J} M(\phi^j)-M(w_n^J)]=0.
	\end{equation}
\end{prop}
\begin{rem}
	In the situation of inhomogeneous NLS,
	we will use  the following vanishing condition
	\begin{equation}\label{vanishing}
		\limsup_{J\to J^*}\limsup_{n\to\infty} \|e^{it\Delta}w_n^J\|_{L_t^\gamma L_x^{\rho,2} (\mathbb{R}\times\mathbb{R}^d)} = 0,
	\end{equation}
	instead of (\ref{vanishing0}).  Indeed, by Strichartz estimate and (\ref{mass-decoupling}), 
	\begin{equation}
		 \|e^{it\Delta }w_n^J\|_{L_t^\infty L_x^2\cap L_t^2L_x^{\frac{2d}{d-2}}}\lesssim   \|w^{J}_n\|_{L^2_x}\lesssim \limsup_{n\rightarrow \infty }  \|u_n\|_{L_x^2}\lesssim  1.\label{6155}
	\end{equation}
	Using Lemma \ref{CZ} to interpolate between  (\ref{vanishing0}) and (\ref{6155}), we obtain  (\ref{vanishing}).  
\end{rem}

\begin{proof}[\textbf{Proof of  Theorem \ref{T:reduction}}.]
We will only prove the  existence of the almost periodic solution,  since the three enemies reduction follows from the same argument of \cite{KTV}, which relies only on the structure of group and the information on the movement of $N(t)$ in Subsection \ref{S:3.2}. 

Let 
\begin{align}
		M_c=:\sup \{m:\text{If } \|u_0\|_{L^2}<m,\ \text{then  (INLS)}  (u_0)\ \text{is global and scatters} \},\notag
\end{align}
 where  INLS($u_0$) denotes to the solution of the equation (\ref{nls}) with initial data  $u_0$.  
In light of the small-data scattering theory in  Proposition  \ref{T:CP}, the failure of Theorem~\ref{T:1} implies  $M_c\in (0,+\infty )$ in the defocusing case; and   $M_c\in(0,M(Q))$ in the focusing case.    In particular, by the definition of  $M_c$, we can find a sequence of initial data $u_{0,n}$ with corresponding global solutions $u_n$ to \eqref{nls} satisfying the following:
\begin{align}
M(u_{0,n})\nearrow M_c\quad\text{and}\quad 	\|u_n\|_{L_t^\gamma L^{\rho,2}_x(\mathbb{R}_\pm  \times \mathbb{R}^d )}\to\infty. \label{616z1}
\end{align}		
Applying  Proposition \ref{P:LPD} to the  bounded sequence  $u_{0,n}$ (passing to a subsequence if necessary), we obtain the  linear profile decomposition decomposition 
\begin{equation}
		u_n = \sum_{j=1}^J (\lambda ^j_n)^{-\frac{d}{2}}e^{ix\cdot \xi_n}[e^{it_n^j\Delta}\phi^j](\frac{x-x_n^j}{\lambda _n^j}) + w_n^J,\label{716w1}
\end{equation}
 with properties  (\ref{orthogonality})--(\ref{vanishing}) hold.   We will establish the following facts: \\
(i) there exists a single profile $\phi$ in the decomposition  (\ref{716w1}):
\begin{equation}
	u_{0,n}(x)=\lambda_n^{-\frac{d}{2}}e^{ix\cdot \xi_n}[e^{it_n\Delta }\phi](\frac{x-x_n}{\lambda_n})+w_n(x).\label{716w2}
\end{equation}
(ii) the    parameters in (\ref{716w2}) obey $(t_n,x_n,\xi_n)\equiv (0,0,0)$; \\
(iii) the remainder in (\ref{716w2}) obeys $w_n\to 0$ strongly in $L^2$.

In particular, items (i)--(iii) imply that
\begin{equation}
	u_{0,n}(x)=\lambda _n^{-\frac{d}{2}}\phi (\frac{x}{\lambda _n})+w_n(x),\qquad w_n(x)\stackrel{L^2}{\longrightarrow}0.\notag
\end{equation}
The solution  $u_c$ to (\ref{nls}) with initial data  $\phi$ will then obey all of 	the conditions appearing in Theorem \ref{T:reduction}.
Therefore, 	to complete the proof of  Theorem \ref{T:reduction}, it sufficies to establish items (i)--(iii).  	

\emph{Proof of (i).}   Suppose by  contradiction that   $J^*\ge2$.   It then follows from   (\ref{mass-decoupling}) that for all  $1\le  j\le J^*$,
	\begin{equation}
		M(\phi^j)< \limsup _{n\rightarrow \infty } M(u_n)\le M_c.\label{213w5}
	\end{equation}
	We now construct scattering solutions to (\ref{nls}) corresponding to each profile. 
	First, if  $t_n^j \equiv 0$, then we take $v^j$ to be the solution to \eqref{nls} with initial data $\phi^j$. This solution scatters due to    (\ref{213w5}) and the definition of  $M_c$.   If instead   $t_n^j\to\pm\infty$, we let $v^j$ be the solution that scatters to $e^{it\Delta}\phi^j$ as $t\to\pm\infty$ (see   Proposition  \ref{T:CP}).
	In either of these cases, we then define
	\begin{equation}
		v_n^j(t,x) = :(\lambda_n^j)^{-\frac{d}2} e^{ix\cdot \xi_n}v^j\left(\tfrac{t}{(\lambda_n^j)^2}+t_n^j,\tfrac{x}{\lambda_n^j}\right).\notag
	\end{equation}

Let 
	\begin{equation} 
		u_n^J(t,x) = :\sum_{j=1}^J v_n^j(t,x) + e^{it\Delta} w_n^J(x).\notag
	\end{equation}	
	By the  construction of  $u_n^J$ and  the decomposition  (\ref{716w1}), we  have 
	\begin{equation}
		\lim_{n\rightarrow \infty }\|u_n^J(0)-u_n(0)\|_{L^2}=0.\notag
	\end{equation}
	Assume for a while that we have proved 
\begin{lem}
	\label{C:615}
	The following limits hold:  
	\begin{eqnarray}
		&& \limsup_{J\to J^*}\limsup_{n\to\infty}\bigl\{ \|u_n^J\|_{L_t^\infty L^2(\mathbb{R} \times \mathbb{R} ^d)} + \| u_n^J\| _{L^\gamma _tL_x^{\rho,2}(\mathbb{R} \times \mathbb{R} ^d)}\bigr\}\lesssim 1\label{615x1}\\
		&&  \limsup_{J\to J^*}\limsup_{n\to\infty} \|(i\partial_t + \Delta)u_n^J -\mu |x|^{-b}|u_n^J|^\frac{4-2b}{d}  u_n^J\|_{L^2_tL_x^{\frac{2d}{d+2},2}(\mathbb{R} \times \mathbb{R} ^d)} = 0. \label{615x2}
	\end{eqnarray}
\end{lem}
\noindent 	Applying the stability result (Proposition~\ref{P:stab}), we derive bounds for the solutions $u_n$ that contradicts    (\ref{616z1}).  
	The proof of Lemma \ref{C:615} is postponed to the end of this Subsection. 
	
	\emph{Proof of (ii).} Having established  $J^*=1$,   (\ref{716w1}) simplifies to (\ref{716w2}),
	where  either $t_n\equiv0$ or  $t_n\rightarrow \pm \infty $, either  $x_n\equiv0$ or  $\frac{|x_n|}{\lambda _n}\rightarrow \infty $ and  either  $\xi_n\equiv 0$ or  $\lambda_n\xi_n\rightarrow\infty $.

	To see that the time shifts must obey $t_n\equiv 0$, we note that if $t_n\to\infty$, then 
	\begin{equation}
		e^{it\Delta }u_{0,n}(x)=\lambda_n^{-\frac{d}{2}}e^{-it|\xi_n|^2}e^{ix\cdot\xi_n}\left(e^{i\frac{t}{\lambda_n^2}\Delta }e^{it_n\Delta }\phi\right)(\frac{x-2\xi_n t-x_n}{\lambda_n}),\notag
	\end{equation}
with  asymptotically vanishing space-time norm
	\begin{equation}
		 \|e^{it\Delta }u_{0,n}\|_{L^\gamma _tL_x^{\rho,2}([0,\infty )\times \mathbb{R} ^d)}\lesssim  \|e^{it\Delta }\phi\|_{L^\gamma _tL_x^{ \rho,2}([t_n,\infty )\times \mathbb{R} ^d)}\rightarrow0\quad\text{as}\quad n\rightarrow0,\notag
	\end{equation}
	 defines good approximate solutions with  uniofrm global space-time bounds for $n$ large.  An application of the stability result would  yield uniform space-time bounds that contradicts    (\ref{616z1}).   The case  $t_n\rightarrow-\infty $ is similar.   Hence (\ref{716w2}) simplifies to 
	 \begin{equation}
	 		u_n=\lambda_n^{-\frac{d}{2}}e^{ix\cdot \xi_n} \phi(\frac{x-x_n}{\lambda_n})+w_n(x), \label{716w4}
	 \end{equation}
	 where   either  $x_n\equiv0$ or  $\frac{|x_n|}{\lambda _n}\rightarrow \infty $ and  either  $\xi_n\equiv 0$ or  $\lambda_n|\xi_n|\rightarrow\infty $.

	 To see that the space  and frequency translation parameters must obey $x_n\equiv\xi_n\equiv 0$, we note that if $\frac{|x_n|}{\lambda _n}\to\infty$ or  $\lambda_n|\xi_n| \rightarrow\infty $, then Proposition~\ref{P:embed} yields global scattering solutions $v_n$ to \eqref{nls} with 
	 $$v_n(0)= e^{i\lambda_n x\cdot\xi_n}\phi(x-\frac{x_n}{\lambda_n}).$$ 
	 By rescalling,   $\lambda_n^{-\frac{d}{2}}v_n(\lambda_n^{-2}t,\lambda_n^{-1}x)$ are also  global scattering solutions to (\ref{nls}) with initial data   $\lambda_n^{-\frac{d}{2}}e^{ix\cdot \xi_n} \phi(\frac{x-x_n}{\lambda_n})$. 
	 Applying the stability result, we derive bounds for the solutions $u_n$ that contradict    (\ref{616z1}).   Hence (\ref{716w4}) simplifies to 
	 \begin{equation}
	 	u_n=\lambda_n^{-\frac{d}{2}}  \phi(\frac{x}{\lambda_n})+w_n(x).\notag
	 \end{equation}

	\emph{Proof of (iii).} 
	Finally, we prove that 
	\begin{equation}
		\|w_n(x)\|_{L^2}\rightarrow0\quad\text{as}\quad  n\rightarrow+\infty .\label{2141}
	\end{equation}
	Indeed, otherwise, by  (\ref{mass-decoupling}), we see that  $M(\phi)<M_c$.  
	Then by the definition of  $M_c$, the solution $ v $ to (\ref{nls}) with initial data  $\phi $ scatters, and 
	\begin{equation}
		v_n(t,x)=:\lambda_n^{-\frac{d}{2}}v\left(\frac{t}{\lambda_n^2},+\frac{x}{\lambda_n}\right)\notag
	\end{equation}
	are good approximate solutions with uniofrm global space-time bounds for $n$ large.
	Applying the stability result, we derive bounds for the solutions $u_n$ that contradicts    (\ref{616z1}).  
\end{proof}
It therefore remains to prove Lemma \ref{C:615}. 
\begin{proof}[\textbf{Proof of Lemma \ref{C:615}}.]
 We give only a description of the proof of the lemma, as the standard  argument  can be found in  \cite[Proposition 5.3]{KV}. 
	The key ingredient is the following decoupling statement for the nonlinear profiles:
	\begin{equation} 
		\lim_{n\to\infty} \|v_n^j v_n^k \|_{L_t^{\frac{\gamma}{2}}L_x^{\frac{\rho}{2},1}\cap L^\infty _tL^1_x(\mathbb{R}\times\mathbb{R}^d)} = 0\quad\text{for}\quad j\neq k,\label{615x3}
	\end{equation}
	which follows from approximation by functions in $C_c^\infty(\mathbb{R}\times\mathbb{R}^d)$ and the use of the orthogonality conditions (\ref{orthogonality}).   
	
	We now prove (\ref{615x1}).   	By  the construction of  $v_n^j$,  
 \begin{align}
 			& \|\sum _{j=1}^{J}v_n^j\|_{L^\gamma _tL_x^{\rho,2}\cap L^\infty _tL^2_x}^2= \|(\sum _{j=1}^{J}v_n^j)^2\|_{L_t^{\frac{\gamma }{2}}L_x^{\frac{\rho}{2},1}\cap L^\infty _tL^1_x},\notag  \\
 			&\lesssim \sum_{j=1}^{J} \|v_n^j\|^{2}_{L_t^\gamma L_x^{\rho,2}\cap L^\infty _tL^2_x}+\sum _{j\neq k} \|v_n^jv_n^k\|_{L_t^{\frac{\gamma }{2}}L_x^{\frac{\rho}{2},1}\cap L^\infty _t L^1_x}\notag\\
 			&\lesssim \sum _{j=1}^J \|\phi^j\|_{L_x^2}^2+   \sum _{j\neq k} \|v_n^jv_n^k\|_{L_t^{\frac{\gamma }{2}}L_x^{\frac{\rho}{2},1}\cap L^\infty _t L^1_x}.\notag
 \end{align}
 The above estimate combined with the vanish (\ref{vanishing}) and the decoupling (\ref{615x3}) imply 
 \begin{equation}
 	 \limsup_{J\to J^*}\limsup_{n\to\infty}  \|u_n^J\|_{L_t^\gamma  L_x^{\rho,2}\cap L_t^\infty L_x^2 }  \lesssim   \limsup_{J\to J^*}\limsup_{n\to\infty}  \left(\sum _{j=1}^J M(\phi^j)+M(w_n^J)\right)\lesssim M_c.\notag
 \end{equation}

	We now focus on the proof of (\ref{615x2}).  To this end, we write $f(z)=:-\mu |z|^\frac{4-2b}{d}  z$  and  $e=:	(i\partial_t + \Delta)u_n^J + |x|^{-b} f(u_n^J)$.  Observe that
\begin{equation}
			e = |x|^{-b}[f(u_n^J) - f(u_n^J - e^{it\Delta} w_n^J)]+|x|^{-b}\biggl[f\bigl(\sum_{j=1}^J v_n^j\bigr) - \sum_{j=1}^J f(v_n^j)\biggr].\notag
\end{equation}
Using  the elementary inequality  $|f(z_1)-f(z_2)|\lesssim  (|z_1|^{\frac{4-2b}{d} }+|z_2|^{\frac{4-2b}{d} })|z_1-z_2|$,  we get 
\begin{equation}
		|e|\lesssim  |x|^{-b}(|u_n^J|^{\frac{4-2b}{d} }+|e^{it\Delta }w_n^J|^{\frac{4-2b}{d} })|e^{it\Delta }w_n^J|+|x|^{-b}\sum _{j\neq k}|v_n^k||v_n^j|^{\frac{4-2b}{d} }.\notag
\end{equation}
	By Lemma \ref{L:nonlinear estimate}, 
\begin{align}
	 &\|  |x|^{-b}(|u_n^J|^{\frac{4-2b}{d} }+|e^{it\Delta }w_n^J|^{\frac{4-2b}{d} })|e^{it\Delta }w_n^J|\| _{L^2_tL_x^{\frac{2d}{d+2},2}}\notag\\
	 &\lesssim   (\|u_n^J\|^\frac{4-2b}{d} _{L^\gamma _tL_x^{\rho,2}} + \|e^{it\Delta }w_n^J\|_{L^\gamma _tL_x^{\rho,2}} ^\frac{4-2b}{d} ) \|e^{it\Delta }w_n^J\|_{L^\gamma _tL_x^{\rho,2}} \notag
\end{align}
	and 
	\begin{equation}
		 \||x|^{-b}\sum _{j\neq k}|v_n^k||v_n^j|^{\frac{4-2b}{d} }\|  _{L^2_tL_x^{\frac{2d}{d+2},2}}\lesssim  \begin{cases}
		 	\sum _{j\neq k} \|v_n^j\|^{\frac{4-2b}{d} -1}_{L_t^\gamma L_x^{\rho,2}}  \|v_n^kv_n^j\|_{L_t^{\frac{\gamma }{2}}L_x^{\frac{\rho}{2},1}},\ &\text{when }\frac{4-2b}{d} \ge1,   \\
		 		\sum _{j\neq k} \|v_n^k\|_{L_t^\gamma L_x^{\rho,2}} ^{1-\frac{4-2b}{d} } \|v_n^kv_n^j\|_{L_t^{\frac{\gamma }{2}}L_x^{\frac{\rho}{2},1}}^\frac{4-2b}{d} ,\ &\text{when }\frac{4-2b}{d} <1.
		 \end{cases}\notag
	\end{equation}
The above estimates,  combined with the vanish  (\ref{vanishing}), the boundedness (\ref{615x1}) and the decoupling  (\ref{615x3}) imply (\ref{615x2}).   
\end{proof}

\subsection{The properties of the almost periodic solution.}\label{S:3.2}
Recalling definition \ref{D:1}, if  $u$ is an almost periodic solution to  (\ref{nls}) on maximal interval  $I$, then there exist  functions $N : I \to \mathbb{R}^+$  and a function $C : \mathbb{R}^+ \to \mathbb{R}^+$ such that (\ref{E:compact1}) and (\ref{E:compact2}) hold.   We now  record some    properties  of  $N(t)$. 

\begin{lem}\label{L:Ndj}
	If  $u$ is an almost periodic  solution to (\ref{nls}) and  $J$ is an inteval with 
	\begin{equation}
		\|u\|_{L^\gamma _tL^{\rho,2}_x(J\times \mathbb{R} ^d)}\le C,\notag 
	\end{equation}
	then for  $t_1,t_2\in J$, 
	\begin{equation}
		N(t_1)\approx _CN(t_2).\notag
	\end{equation}
\end{lem}
\begin{defn}
	Assume  $u$ satisfies the assumption of Lemma \ref{L:Ndj}. 	Divide $J$ into consecutive intervals $J_k$ such that
	\begin{equation}
		\|u\|_{L^\gamma _tL_x^{\rho,2}(J_k \times \mathbb{R}^d)} = 1.\notag
	\end{equation}
	We call these intervals \emph{the intervals of local constancy}.
\end{defn}

\begin{lem}\label{L:6131}
	Suppose  $u$ is an almost periodic  solution to (\ref{nls})  with  $N(t)\le1$.    Let   $J_k$ be the  local constant interval and 
	\begin{equation}
		N(J_k)=:\sup _{J_k}N(t).\notag
	\end{equation}
	Then 
	\begin{equation}
		\sum _{J_k}N(J_k) \approx \int _JN(t)^3dt.\notag
	\end{equation}    
\end{lem}
The proof of Lemma \ref{L:Ndj} and Lemma \ref{L:6131} follows from \cite[Corollary 3.6]{KTV} and \cite[Lemma 2.5]{Dodson4}, respectively, which  are both the consequences of the Cauchy theory and the  compactness  (\ref{E:compact1})--(\ref{E:compact2}).  

Therefore, possibly after modifying the  $C(\eta)$ in  (\ref{E:compact1}), (\ref{E:compact2}) by a fixed constant, we can choose  $N(t):I\rightarrow (0,\infty )$  such that   
\begin{equation}
	|\frac{d}{dt}N(t)|\lesssim  N(t)^3.\label{E:N}
\end{equation}
\begin{lem}[Spacetime bound]\label{L:6132}
	If  $u(t,x)$  is an  almost periodic  solution to (\ref{nls})  on an interval  $J$. Then for any Schr\"odinger admissible pair  $(q,r)$ with  $q<\infty $
	\begin{equation}
		\int _JN(t)^2dt\lesssim   \|u\|_{L_{t} ^ q  L_x^{r,2}(J\times \mathbb{R} ^d)}^q \lesssim 1+\int _JN(t)^2dt. \label{531w1}
	\end{equation}
\end{lem}
\begin{proof}
	\noindent\textbf{Step 1:} 	We first prove the upper bound
	\begin{equation}
		\|u\|_{L_{t} ^ q  L_x^{r,2}(J\times \mathbb{R} ^d)}^q \lesssim 1+\int _JN(t)^2dt.\label{6151}
	\end{equation}
	Let \( 0 < \eta < 1 \)  to be  determined later and partition \( J \) into subintervals \( I_j \)   such that 
	\begin{equation}
		\int_{I_j} N(t)^2 \, dt \leq \eta,\qquad 1\le j\le \eta ^{-1} \left(1+\int _JN(t)^2dt\right)  .\notag
	\end{equation}
For each \( j \), we may choose \( t_j \in I_j \)  such that 
	\begin{equation}
		N(t_j)^2 |I_j| \leq 2\eta. \notag
	\end{equation}
	Let  $(\gamma ,\rho)$ be defined by (\ref{E:gamma}), satisfying (\ref{zb}).  By Strichartz' inequality, we have the following estimates on the spacetime slab \( I_j \times \mathbb{R}^d \):
	\begin{align}
		\|u\|_{L_{t} ^ \gamma   L_x^{\rho,2}}  &\lesssim  \|e^{i(t-t_{j})\Delta }u(t_j)\| _{L_{t} ^ \gamma   L_x^{\rho,2}} + \||x|^{-b}|u|^{\frac{4-2b}{d} }u\|_{L_t^2L_x^{\frac{2d}{d+2},2}}\notag\\
		&\lesssim   \|u_{\ge C(\eta)N(t_j)}\|_{L^2_x}+ \|e^{i(t-t_j)\Delta }u_{\le C(\eta)N(t_j)}\|  _{L^\gamma _tL_x^{\rho,2}}+ \||x|^{-b}\|_{L_x^{\frac{d}{b},\infty }}  \|u\|^{\frac{4-2b}{d} +1}_{L_t^{\gamma }L_x^{\rho,2}}\notag\\
		&\lesssim    \|u_{\ge C(\eta)N(t_j)}\|_{L^2_x}+|I_j|^{\frac{1}{\gamma  }}N(t_j)^{d(\frac{1}{2}-\frac{1}{\rho})} \|u(t_j)\|_{L_x^2}+  \|u\|^{\frac{4-2b}{d} +1}_{L_t^{\gamma }L_x^{\rho,2}}, \label{6152}
	\end{align}
	where we  used Holder’s and Bernstein’s inequalities  in the last step.   By (\ref{E:compact2}) and 
	\begin{equation}
		|I_j|^{\frac{1}{\gamma }}N(t_j)^{d(\frac{1}{2}-\frac{1}{\rho})}=(|I_j|N(t_j)^2)^{\frac{1}{\gamma }}\le (2\eta )^{\frac{1}{\gamma }},\notag
	\end{equation}
	the  first two terms in (\ref{6152}) can be made arbitrarily small,  as long as  $\eta $ is  sufficiently small. Thus,  by the usual bootstrap argument we obtain
	\begin{equation}
		\|u\|_{L^\gamma _tL_x^{\rho,2}(I_j\times \mathbb{R} ^d)}\lesssim 1.\notag 
	\end{equation}
	Again by Strichartz estimate,
	\begin{equation}
		\|u\|_{L^q_tL_x^{r,2}(I_j\times \mathbb{R} ^d)}\lesssim   \|e^{i(t-t_{j})\Delta }u(t_j)\| _{L_{t} ^ q   L_x^{r,2}(I_j\times \mathbb{R} ^d)} +\||x|^{-b}\|_{L_x^{\frac{d}{b},\infty }}  \|u\|^{\frac{4-2b}{d} +1}_{L_t^{\gamma }L_x^{\rho,2}(I_j\times \mathbb{R} ^d)}\lesssim 1.\notag 
	\end{equation}
Summing over all intervals  $I_j$, we obtain (\ref{6151}).
	
	 \noindent\textbf{Step 2:} Now we prove  the lower bound
	\begin{equation}
		\|u\|_{L^q_tL_x^{r,2}(J\times \mathbb{R} ^d)}\gtrsim_u\left(\int _JN(t)^2dt\right) ^{\frac{1}{q}}.\label{6153}
	\end{equation}
	Using (\ref{E:compact1}) and  choosing  $\eta$ sufficiently small depending on  $M(u)$, we can guarantee that
	\begin{equation}
		\int _{|x|\le C(\eta)N(t)^{-1} }|u(t,x)|^2dx\gtrsim_u1\quad\text{for all}\quad  t\in J.\label{6154}
	\end{equation}
	On the other hand, a simple application of  H\"older's inequality yields 
	\begin{equation}
		 \|u\|_{L_x^{r,2}}\gtrsim \|u\|_{L_x^r}\gtrsim_u N(t)^{d(\frac{1}{2}-\frac{1}{r})} \left(\int _{|x|\le C(\eta)N(t)^{-1}}|u(t,x)|^2dx\right)^{\frac{1}{2}}.\notag
	\end{equation}
	Thus, using (\ref{6154}) and integrating over  $J$ we derive (\ref{6153}).   
\end{proof}

Finally, we prove the following estimates, which will be used in Section \ref{S:6} to exclude the existence of a quasi-soliton.
\begin{lem}\label{L:693}
	Let  $u(t,x)$ be an almost periodic solution to (\ref{nls}) on the interval  $[0,T]$,  $K=\int_0^TN(t)^3dt$,  $I=P_{\le K}$ and   $\phi $ be a nongative function  such that 
	\begin{equation}
		\phi(x)=1, \ |x|\le 1 \qquad \text{supp } \phi \subset \{x:|x|\le2\}.\notag
	\end{equation}
Then there exists a constant  $R_0>0$	 such that the following estimates hold for all  $R,K>R_0$:\\
	(i) 
	\begin{equation}
		\int_0^TN(t)\int_{\mathbb{R} ^d} \phi^2 (\frac{rN(t)}{R}) |\nabla Iu|^2dxdt\gtrsim K,\label{613x2}
	\end{equation}
	(ii) 
	\begin{equation}
		\int _0^T N(t)\int_{\mathbb{R} ^d}(1-\phi (\frac{rN(t)}{R}))   |x|^{-b}|Iu|^{\frac{4-2b}{d} +2}dxdt=o_R(1)K.\label{525ww1}
	\end{equation}
\end{lem}
\begin{proof}
	Let  $J_k$ be local constant interval.   By Strichartz estimate, 
	\begin{equation}
		\|u\|_{L^2_tL_x^{\frac{2d}{d-2},2}(J_k\times \mathbb{R} ^d)}\lesssim 1. \label{613x1}
	\end{equation}    
	On the other hand, by  the compactness  (\ref{E:compact1}) and (\ref{E:compact2}), when  $R$  and  $K$ sufficiently  large, 
	\begin{equation}
		\|(1-\phi (\frac{rN(t)}{R}))u\|_{L^\infty _tL^2_x(J_k\times \mathbb{R} ^d)}+ \|(1-I)u\|_{L^\infty _tL^2_x(J_k\times \mathbb{R} ^d)} \ll1.\label{732}
	\end{equation}
	Interpolating the above inequality with (\ref{613x1}),  we obtain 
	\begin{equation}
		 \|\left(1-\phi (\frac{rN(t)}{R})I\right)u\| _{L^\gamma _tL_x^{\rho,2}(J_k\times \mathbb{R} ^d)}\ll 1.\notag
	\end{equation}
	Furthermore, since   $ \|u\|_{L^\gamma _tL_x^{\rho,2}(J_k\times \mathbb{R} ^d)}=1 $ and   $\|u\|_{L^\infty  _tL_x^{ 2}(J_k\times \mathbb{R} ^d)}\lesssim  \|u_0\|_{L^2}  $, we have 
	\begin{equation}
		\|\phi (\frac{rN(t)}{R})Iu\|_{L^2_tL_x^{\frac{2d}{d-2},2}(J_k\times \mathbb{R} ^d)}\gtrsim1.\notag
	\end{equation}
	Therefore,  by  Sobolev's embedding (Lemma \ref{L:sobolev}), (\ref{531w1}) and (\ref{613x1}),  
	\begin{align}
		1&\lesssim  \|\phi (\frac{rN(t)}{R})Iu\|_{L^2_tL_x^{\frac{2d}{d-2},2}(J_k\times \mathbb{R} ^d)}^2\lesssim  \|\nabla \left(\phi(\frac{rN(t)}{R})Iu\right)\|  _{L_t^2L_x^2(J_k\times \mathbb{R} ^d)}^2\notag\\
		&\lesssim  \int _{J_k\times \mathbb{R} ^d}\phi^2(\frac{rN(t)}{R})|\nabla Iu|^2dxdt+\int _{J_k\times \mathbb{R} ^d}\frac{N(t)^2}{R^2}|Iu|^2dxdt\notag\\
		&\lesssim \int _{J_k\times \mathbb{R} ^d}\phi^2(\frac{rN(t)}{R})|\nabla Iu|^2dxdt+\frac{1}{R^2}.\notag
	\end{align}
	Hence, for  $R$ and  $K$  sufficiently large, we can apply 	 Lemma \ref{L:6131}  to deduce  (\ref{613x2}):
	\begin{align}
		&\int_0^TN(t)\int\phi (\frac{rN(t)}{R})|\nabla Iu|^2dxdt\gtrsim \sum_k N(J_k)\int _{J_k\times \mathbb{R} ^d}\phi^2(\frac{rN(t)}{R})|\nabla Iu|^2dxdt\notag\\
		&\gtrsim \sum _kN(J_k)\approx \int_0^TN(t)^3dt=K.\notag
	\end{align}

	We now prove (\ref{525ww1}).   Let  $(\gamma ,\rho)$ be defined by (\ref{E:gamma}), satisfying (\ref{zb}).  By   H\"older's inequality,      
	\begin{align}
		&\int _{J_k}\int_{\mathbb{R} ^d} (1-\phi (\frac{rN(t)}{R}))|x|^{-b}|Iu|^{\frac{4-2b}{d} +2}dxdt \label{731}\\
		&\lesssim   \| (1-\phi (\frac{rN(t)}{R})) Iu\|_{L^\gamma _tL^{\rho,2}_x(J_k\times \mathbb{R} ^d)}^{\frac{4-2b}{d} +1} \|u\|_{L^2_tL_x^{\frac{2d}{d-2},2}(J_k\times \mathbb{R} ^d)}.\notag  
	\end{align}
	Interpolating between (\ref{613x1}) and (\ref{732}), we see that  $(\ref{731})=o_R(1)$.  
	Hence,  
	\begin{align}
		&\int _0^TN(t)\int_{\mathbb{R} ^d}(1-\phi (\frac{rN(t)}{R}))  |x|^{-b}|Iu|^{\frac{4-2b}{d} +2}dxdt \notag\\
		&\lesssim  \sum _kN(J_k) \int _{J_k} \int_{\mathbb{R} ^d}(1-\phi (\frac{rN(t)}{R})) |x|^{-b}|Iu|^{\frac{4-2b}{d} +2}dxdt 
		\notag\\
		&	\lesssim \sum _k N(J_k)o_R(1) \lesssim \int_0^TN(t)^3dto_R(1) \lesssim K\cdot o_R(1).\notag
	\end{align}
	Therefore, we complete the proof of Lemma \ref{L:693}.   
\end{proof}

\section{Long time Strichartz estimate}\label{S:4}
In this section we prove the long time Strichartz estimates in  Proposition  \ref{T:long time SZ}.  These estimates will  be
used in Section \ref{S:5} and Section \ref{S:6} to preclude the rapid cascade   $\int_0^\infty N(t)^3dt<\infty $  and the quasi-soliton, respectively. 

We start with the following Lemma. 
\begin{lem}\label{L:long time SZ}
	Let  $d\ge 3,0<b<\min \left\{ 2,\frac{d}{2} \right\}$ and  $\mu=\pm1$.
	Suppose  $J\subset [0,\infty )$ is compact,     $u$ is the almost periodic  solution given by  Theorem  \ref{T:reduction}, and  $\int _JN(t)^3dt=K$.  Let  $J_k$ be the local constant interval. 
	Then for any  $0<\alpha <\min \left\{ \frac{d-2b}{2}+1,\frac{4-2b}{d}+1 \right\}$, the following   estimate holds:\\
(i) When  $\frac{4-2b}{d}\le1$: 
	\begin{align}
		 \|u_{>N}\|_{L^2_tL_x^{\frac{2d}{d-2},2}(J\times \mathbb{R} ^d)}\lesssim _\alpha& \sigma _{J}\big(\tfrac{N}{2}\big)+\sum_{M\le \eta N} \big(\tfrac{M}{N}\big)^\alpha \|u_{>M}\|_{L^2_tL_x^{\frac{2d}{d-2},2}(J\times \mathbb{R} ^d)}\notag\\
		 &+\big(\tfrac{K}{\eta N}\big)^{\frac{2-b}{d}}\sup _{J_k} \|u_{>\eta N}\|_{S^0_*(J\times \mathbb{R} ^d)}^{\frac{4-2b}{d}}\|u_{>\eta N}\|_{L^2_tL_x^{\frac{2d}{d-2},2}(J\times \mathbb{R} ^d)}^{1-\frac{4-2b}{d}}.   \label{5171}
	\end{align}
\noindent (ii)When  $\frac{4-2b}{d}>1$: 
	\begin{align}
			\|u_{>N}\|_{L^2_tL_x^{\frac{2d}{d-2},2}(J\times \mathbb{R} ^d)}\lesssim _\alpha &\sigma _{J}\big(\tfrac{N}{2}\big)+\sum_{M\le \eta N} \big(\tfrac{M}{N}\big)^\alpha \|u_{>M}\|_{L^2_tL_x^{\frac{2d}{d-2},2}(J\times \mathbb{R} ^d)}\notag\\
			&+\big(\tfrac{K}{\eta N}\big)^{\frac{1}{2}}\sup _{J_k} \|u_{>\eta N}\|_{S^0_*(J\times \mathbb{R} ^d)},  \label{5172}
	\end{align}
	where  $\sigma_J(N)$,   $S_*^0(J\times \mathbb{R} ^d)$ were defined in (\ref{691}) and (\ref{523x1}), respectively.  
 \end{lem}

\begin{proof}
	Define a cutoff  $\chi (t,\cdot)\in C_0^\infty (\mathbb{R} ^d)$ in physical spcae, 
	\begin{equation}
		\chi(t,x)=\begin{cases}
			1,\qquad |x|\le \frac{C_0}{N(t)};\\
			0,\qquad |x|>\frac{2C_0}{N(t)},
		\end{cases}\notag
	\end{equation} 
	where we choose  $C_0=C_0(\eta)$ such that by (\ref{E:compact1}) and (\ref{E:compact2}), 
	\begin{equation}
		 \|u_{>C_0(\eta)N(t)} \|_{L^\infty _tL^2_x}^{\frac{4-2b}{d}}+ \|(1-\chi (t)) u_{\le C_0(\eta)N(t)} \|_{L^\infty _tL^2_x}  ^{\frac{4-2b}{d}} \le \eta^2.\label{6121}
	\end{equation}
	Let    $F(u)=:\mu |x|^{-b}|u|^{\frac{4-2b}{d}}u$.  Then   
	\begin{align}
	    \|P_{>N}&F(u)\|_{L^2_tL_x^{\frac{2d}{d+2},2}(J\times \mathbb{R} ^d)}\lesssim   \|P_{>N}F(u_{\le \eta N})\|_{L^2_tL_x^{\frac{2d}{d+2},2}(J\times \mathbb{R} ^d)}\notag\\
	     &+ \||x|^{-b}u_{>\eta N}(1-\chi(t))|u_{\le C_0N(t)}|^{\frac{4-2b}{d}}\|_{L^2_tL_x^{\frac{2d}{d+2},2}(J\times \mathbb{R} ^d)} \label{7141}\\
	    & + \||x|^{-b}u_{>\eta N}|u_{>C_0 N(t)}|^{\frac{4-2b}{d}} \|_{L_t^2L_x^{\frac{2d}{d+2},2}(J\times \mathbb{R} ^d)} \label{7142}\\
	    &+ \||x|^{-b}\chi (t)u_{>\eta N}|u_{\le C_0 N(t)}|^{\frac{4-2b}{d}} \|_{L_t^2L_x^{\frac{2d}{d+2},2}(J\times \mathbb{R} ^d) }.\label{7143}
	\end{align}
Fix   $\varepsilon _0>0$ sufficiently small  such that  $\alpha+\varepsilon _0<\min \left\{ \frac{d-2b}{2}+1,\frac{4-2b}{d}+1 \right\}$.  By Bernstein and Lemma \ref{L:nonlinear estimate}, we have the following estimate on   $J\times \mathbb{R} ^d$,
	\begin{align}
		 &\|P_{>N}F(u_{\le \eta N})\|_{L^2_tL_x^{\frac{2d}{d+2},2}}\lesssim _\alpha\frac{1}{N^{\alpha+\varepsilon _0}}  \||\nabla |^{s+\varepsilon _0}(|x|^{-b}|u_{\le\eta N}|^{\frac{4-2b}{d}} u_{\le\eta N})\|_{L^2_tL_x^{\frac{2d}{d+2},2}}\notag\\
		 &\lesssim _\alpha \frac{1}{N^{\alpha+\varepsilon _0}} \|u_{\le\eta N}\|_{L^\infty _tL^2_x}^{\frac{4-2b}{d}} \||\nabla |^{\alpha+\varepsilon _0}u_{\le \eta N}\|_{L^2_tL_x^{\frac{2d}{d-2},2}}  +\frac{1}{N^{\alpha+\varepsilon _0}} \||\nabla |^{\frac{2(\alpha+\varepsilon _0)}{\gamma } }u_{\le \eta N}\|^{\frac{\gamma }{2}}_{L^\gamma _tL_x^{ \rho,2}}\notag\\
		 &=:I_{1}+I_{2}.\notag   
	\end{align}
	By Bernstein's inequality (Lemma \ref{L:Bernstein}),  
	\begin{equation}
		I_{1}\lesssim  \frac{1}{N^{\alpha+\varepsilon _0}}\sum_{M\le \eta N} M^{\alpha+\varepsilon _0} \|u_{>M}\|_{L^2_tL_x^{\frac{2d}{d-2},2}}\lesssim  \sum _{M\le \eta N}(\frac{M}{N})^\alpha   \|u_{>M}\|_{L^2_tL_x^{\frac{2d}{d-2},2}(J\times \mathbb{R} ^d)}.\notag
	\end{equation}
 	Again by Bernstein and Young's  inequality, 
 	\begin{align}
 		I_{2}&\lesssim  \frac{1}{N^{\alpha+\varepsilon _0}} \left(\sum _{M\le \eta N} \left(M^\alpha \|P_{>M }u\|_{L_t^\gamma L_x^{\rho,2}}^{\frac{\gamma }{2}}\right) ^{\frac{2 }{\gamma }} M^{\frac{2}{\gamma }\varepsilon _0}\right)^{\frac{\gamma }{2}}\notag\\
 		&\lesssim  \frac{1}{N^{\alpha+\varepsilon _0}}\sum_{M\le \eta N} M^\alpha \|P_{>M }u\|_{L_t^\gamma L_x^{\rho,2}}  ^{\frac{\gamma }{2}} (\eta N)^{\varepsilon _0} 
 		\lesssim  \sum_{M\le \eta N} (\frac{M}{N})^\alpha \|P_{>M}u\|_{L^2_tL_x^{\frac{2d}{d-2},2}(J\times \mathbb{R} ^d)}. \notag
 	\end{align} 
Hence, 
	\begin{equation}
		 \|P_{>N}F(u_{\le\eta N})\|_{L_t^2L_x^{\frac{2d}{d+2},2}(J\times \mathbb{R} ^d)}\lesssim _s\sum _{M\le \eta N}(\frac{M}{N})^\alpha \|u_{>M}\|_{L_t^2L_x^{\frac{2d}{d-2},2}(J\times \mathbb{R} ^d)}.\label{517w1}
	\end{equation}
	
	Next, we estimate (\ref{7141}) and (\ref{7142}).  By  (\ref{6121}) and H\"older's inequality,   on     $J\times \mathbb{R} ^d$,
	\begin{align}
& (\ref{7141})+(\ref{7142}) \notag\\
& \lesssim   \|u_{>\eta N}\|_{L_t^2L_x^{\frac{2d}{d-2},2}}[  \|u_{>C_0(\eta)N(t)} \|_{L^\infty _tL^2_x}^{\frac{4-2b}{d}}+ \|(1-\chi (t)) u_{\le C_0(\eta)N(t)} \|_{L^\infty _tL^2_x}  ^{\frac{4-2b}{d}}  ]\notag\\
&\lesssim  \eta ^2  \|u_{>\eta N}\|_{L^2_xL_x^{\frac{2d}{d-2},2}},\notag  
	\end{align} 
which  is  controlled by   (\ref{517w1}).

	Finally, by H\"older's inequality, the last term can be estimated by  
	\begin{equation}
	(\ref{7143})\lesssim 	\| \chi (t)u_{>\eta N} |u_{\le C_0N(t)}|^{\frac{4-2b}{d}}\|_{L^2_tL_x^{\frac{2d}{d+2-2b},2}(J\times \mathbb{R} ^d)}.\notag 
	\end{equation}

	\noindent \textbf{When  $\frac{4-2b}{d}>1$:}  By the bilinear estimates (\ref{E:bilinear SZ}) and  (\ref{614ww1}), 
	\begin{align}
			&\| \chi (t) u_{>\eta N}  |u_{\le C_0N(t)}|^{\frac{4-2b}{d}}\|_{L^2_tL_x^{\frac{2d}{d+2-2b},2}(J_k\times \mathbb{R} ^d)}\notag\\
			&\lesssim   \|u_{>\eta N}u_{\le C_0 N(t)}\|_{L^2_{t,x}(J_k\times \mathbb{R} ^d)} \|\chi (t)\|_{L^\infty _tL^{\frac{2d}{d-2},2}_x(J_k\times \mathbb{R} ^d)} \|u_{\le C_0N(t) }\|_{L^\infty _tL_x^2 (J_k\times \mathbb{R} ^d)}^{\frac{4-2b}{d}-1}   \notag\\
		& \lesssim  \frac{N(J_k)^{\frac{d-1}{2}}}{(\eta N)^{\frac{1}{2}}} \|u_{>\eta N}\|_{S^0_*(J_k\times \mathbb{R} ^d)} N(J_k)^{-\frac{d-2}{2}} 
		\lesssim  \left(\frac{N(J_k)}{\eta N}\right)  ^{\frac{1}{2}} \|u_{>\eta N}\|_{S^0_*(J_k\times \mathbb{R} ^d)} .\notag 
	\end{align}
	Summing over all subintervals  $J_k$, and using Lemma \ref{L:6131},
	\begin{equation}
\| \chi (t)u_{>\eta N} |u_{\le C_0N(t)}|^{\frac{4-2b}{d}}\|_{L^2_tL_x^{\frac{2d}{d+2-2b},2}(J\times \mathbb{R} ^d)}\lesssim  (\frac{K}{\eta N})^{\frac{1}{2}}(\sup _{J_k}  \|u_{>\eta N}\|_{S^0_*(J_k\times \mathbb{R} ^d)} ).\notag
	\end{equation}

	\noindent \textbf{When  $\frac{4-2b}{d}\le1$:}  By H\"older's inequality,
	\begin{align}
			&\|u_{>\eta N} \chi (t) |u_{\le C_0N(t)}|^{\frac{4-2b}{d}}\|_{L^2_tL_x^{\frac{2d}{d+2-2b},2}(J\times \mathbb{R} ^d)}\notag\\
			& \lesssim  \|\chi(t)|u_{>\eta N}u_{\le C_0N(t)}|^{\frac{4-2b}{d}}\|_{L_t^{\frac{d}{2-b}}L_x^{p,2}(J\times \mathbb{R} ^d)} \|u_{>\eta N}\|_{L^2_xL_x^{\frac{2d}{d-2},2}(J\times \mathbb{R} ^d)}^{1-\frac{4-2b}{d}},\notag  
	\end{align}
	where 
	\begin{equation}
		\frac{1}{p}=:\frac{d+2-2b}{2d}-(1-\frac{4-2b}{d})\frac{d-2}{2d}=\frac{4-2b}{d}\frac{1}{2}+\frac{(2-b)(d-2)}{d^2}.\notag
	\end{equation}
Applying H\"older's inequality and bilinear estimate (\ref{E:bilinear SZ}), we have  
\begin{align}
				&\|u_{>\eta N} \chi (t) |u_{\le C_0N(t)}|^{\frac{4-2b}{d}}\|_{L^{\frac{d}{2-b}}_tL_x^{p,2}(J_k\times \mathbb{R} ^d)}\notag\\
				&\lesssim    \|u_{>\eta N}u_{\le C_0 N(t)}\|_{L^2_{t,x}(J_k\times \mathbb{R} ^d)}^{\frac{4-2b}{d}} \|\chi (t)\|_{L^\infty _tL^{\frac{d^2}{(2-b)(d-2)},2}_x(J_k\times \mathbb{R} ^d)}\notag\\
					& \lesssim \left(\frac{N(J_k)^{\frac{d-1}{2}}}{(\eta N)^{\frac{1}{2}}}\right)^{\frac{4-2b}{d}}  \|u_{>\eta N}\|_{S^0_*(J_k\times \mathbb{R} ^d)}^{\frac{4-2b}{d}}  \|u_{\le C_0N(t)}\|_{S^0_*(J_k\times \mathbb{R} ^d)}^{\frac{4-2b}{d}} N(J_k)^{-\frac{(2-b)(d-2)}{d}} \notag\\
				&\lesssim  \left(\frac{N(J_k)}{\eta N}\right) ^{\frac{2-b}{d}}  \|u_{>\eta N}\|_{S^0_*(J_k\times \mathbb{R} ^d)}^{\frac{4-2b}{d}}.\notag 
\end{align}
Again summing over all subintervals, and using Lemma \ref{L:6131},
\begin{align}
		&\| \chi (t)u_{>\eta N} |u_{\le C_0N(t)}|^{\frac{4-2b}{d}}\|_{L^2_tL_x^{\frac{2d}{d+2-2b},2}(J\times \mathbb{R} ^d)}\notag\\
		&\lesssim  \left(\frac{K}{\eta N}\right) ^{\frac{2-b}{d}}  (\sup _{J_k}  \|u_{>\eta N}\|_{S^0_*(J_k\times \mathbb{R} ^d)} )^{\frac{4-2b}{d}} \|u_{>\eta N}\|_{L^2_tL_x^{\frac{2d}{d-2},2}(J\times \mathbb{R} ^d)}^{1-\frac{4-2b}{d}}.\notag 
\end{align}
This completes the proof of Lemma \ref{L:long time SZ}.    
\end{proof}
Notice that in the proof of Lemma \ref{L:long time SZ}, we have also proved the following estimate, which will be used in Section \ref{S:5} to establish the additional regularity of the rapid cascade.  
\begin{cor}
	For $\frac{1}{2}\le s<\min \left\{ \frac{d-2b}{2}+1,\frac{4-2b}{d}+1 \right\}$, 
	\begin{equation}
		\|P_{>N}F(u)\|_{L^2_tL_x^{\frac{2d}{d+2},2}(J\times \mathbb{R} ^d)}\lesssim_s  \frac{K^{1/2}}{N^{s}} \|u\|_{L^\infty _t \dot H^{s-1/2}_x}+\sigma _J (\frac{N}{2}). \label{5181}
	\end{equation} 
\end{cor}

Now we utilize Lemma \ref{L:long time SZ} to prove   Proposition  \ref{T:long time SZ}. 
\begin{proof}[\textbf{Proof of   Proposition  \ref{T:long time SZ}}.]
	Fix  $J$. Let  
	\begin{equation}
		f(N)=: \|u_{> N}\|_{L^2_tL_x^{\frac{2d}{d-2},2}(J\times \mathbb{R} ^d)}.\notag 
	\end{equation}
	\noindent \textbf{(i)The case  $\frac{4-2b}{d}>1$:}  By Lemma \ref{L:long time SZ} ($\alpha =1$) and (\ref{614ww1}), 
	 \begin{equation}
	 	f(N)\lesssim  \frac{1}{\eta ^{1/2}}[(\frac{K}{N})^{1/2}+\sigma_J(\frac{N}{2})]+\sum_{M\le \eta N}\frac{M}{N}f(M).\notag
	 \end{equation}
	 Let   $\varepsilon >0$ and 
	 \begin{equation}
	 	c_\varepsilon (J)=:\sup _{N}\frac{f(N)}{(\frac{K}{N})^{1/2}+\sigma _J(\frac{N}{2})+\varepsilon },\notag
	 \end{equation}
	 taking the supremum over all dyadic integers   $N$.    
	 By (\ref{531w1}) and  $N(t)\le1$, 
	 \begin{equation}
	 	c_\varepsilon (J)\lesssim  \frac{(1+\int _JN(t)^2dt)^{\frac{1}{2}}}{\varepsilon }\lesssim \frac{(1+|J|)^{\frac{1}{2}}}{\varepsilon }<+\infty .\notag
	 \end{equation}
	 Hence
	 \begin{align}
	 		f(N)&\lesssim   \frac{1}{\eta ^{1/2}}[(\frac{K}{N})^{1/2}+\sigma_J(\frac{N}{2})] +c_\varepsilon (J)\sum _{M\le \eta N}\frac{M}{N}[(\frac{K}{M})^{1/2}+\sigma _J(\frac{M}{2})+\varepsilon ]\notag\\
	 		 &\lesssim  \frac{1}{\eta ^{1/2}}[(\frac{K}{N})^{1/2}+\sigma_J(\frac{N}{2})] +c_\varepsilon (J) [\eta ^{1/2} (\frac{K}{N})^{1/2}+\eta ^{2/3}\sigma_J(\frac{N}{2})+\eta \varepsilon ],\notag 
	 \end{align}
	 where we used (see \cite[Definition 3.3]{Dodson1})
	 \begin{equation}
	 	\sigma _{J}(N)\le  \sigma _J(M)\le (\frac{N}{M})^{1/3}\sigma _J(N).\notag
	 \end{equation}
	  Therefore 
	  \begin{equation}
	  	c_\varepsilon (J)\lesssim  \frac{1}{\eta ^{1/2}}+\eta^{1/2}c_\varepsilon (J).\notag
	  \end{equation}
	  Fixing  $\eta>0$  sufficiently small, 
	  \begin{equation}
	  	c_\varepsilon (J)\lesssim  \frac{1}{\eta ^{1/2}},\notag
	  \end{equation} 
which implies (\ref{long time 1})	as the  bound is independent of  $\varepsilon >0$.  
	  
	  Next we prove (\ref{long time 2}).  Let  $\frac{1}{2}\le s<\min \left\{ \frac{d-2b}{2}+1,\frac{4-2b}{d}+1 \right\}$ and fix  $s_0<s<\min  \{ \frac{d-2b}{2}+1,\frac{4-2b}{d}+1 \}$.    By  Lemma \ref{L:long time SZ} ($\alpha =s_0$), Bernstein and (\ref{614ww2}),  
	  \begin{align}
	  	f(N)&\lesssim  \frac{1}{\eta ^{1/2}}(\frac{K}{N})^{1/2}\sup _{J_k}  \|u_{>\eta N}\|_{S_*^0(J_k\times \mathbb{R} ^d)}+\sigma _J(\frac{N}{2})+\sum _{M\le \eta N} (\frac{M}{N})^{s_0}f(M)\notag\\
	  	&\lesssim _s \frac{1}{\eta ^{s}}\frac{K^{1/2}}{N^{s}} \|u\|_{L^\infty _t \dot H^{s-1/2}_x(J\times \mathbb{R} ^d)}+\sigma _J(\frac{N}{2})+\sum _{M\le \eta N} (\frac{M}{N})^{s_0}f(M).\notag   
	  \end{align}
	  Let   $\varepsilon >0$ and 
	  \begin{equation}
	  	c_{\varepsilon }(J,s)=:\sup _N \frac{f(N)}{\frac{K^{1/2}}{N^{s}}   \|u\|_{L^\infty _t \dot H^{s-1/2}_x(J\times \mathbb{R} ^d)}+\sigma _J(\frac{N}{2}) +\varepsilon }<+\infty .\notag
	  \end{equation}
	  Then 
	  \begin{align}
	  	f(N)&\lesssim_s  \frac{1}{\eta^{s}}\left[ \frac{K^{1/2}}{N^{s}} \|u\|_{L^\infty _t \dot H^{s-1/2}_x}+\sigma _J(\frac{N}{2}) \right]\notag\\
	  	&+c_{\varepsilon }(J,s)\sum _{M\le \eta N}(\frac{M}{N})^{s_0} [\frac{K^{1/2}}{M^{s}} \|u\|_{L^\infty _t \dot H^{s-1/2}_x(J\times \mathbb{R} ^d)}+\sigma _J(\frac{M}{2}) +\varepsilon ]\notag\\
	  	&\lesssim_s \left(\frac{1}{\eta ^{s}}+c_{\varepsilon }(J,s)(\eta ^{s_0-s}+\eta ^{s_0-1/3})\right)\left[ \frac{K^{1/2}}{N^{s}} \|u\|_{L^\infty _t \dot H^{s-1/2}_x}+\sigma _J(\frac{N}{2}) + \varepsilon \right].\notag
	  \end{align}
	  Therefore, 
	  \begin{equation}
	  	c_{\varepsilon }(J,s)\lesssim _{s}\frac{1}{\eta ^{s}}+c_{\varepsilon }(J,s)(\eta ^{s_0-s}+\eta ^{s_0-1/3}).\notag
	  \end{equation}
	  Again fixing  $\eta>0$  sufficiently small, 
	  \begin{equation}
	  	c_{\varepsilon }(J,s)\lesssim_s  \frac{1}{\eta ^{s}},\notag
	  \end{equation} 
	 which  proves (\ref{long time 2}) as the  bound is independent of  $\varepsilon >0$.   
	  
	  \noindent \textbf{(ii)The case  $\frac{4-2b}{d}\le1:$} 
	  Using Young's inequality to  transform the last term in  (\ref{5171}) into a linear term, we obtain
	  \begin{align}
	  	f(N)&\lesssim \sigma _J(\frac{N}{2})+\sum _{M\le \eta N}\frac{M}{N} f(M)+(\frac{K}{\eta N})^{\frac{2-b}{d}}f(\eta N)^{1-\frac{4-2b}{d}}\notag\\
	  	& \lesssim  (\frac{K}{\eta N})^{1/2}+\sigma _J(\frac{N}{2})+\sum _{M\le \eta N}\frac{M}{N}f(M).\notag
	  \end{align} 
	  Proceeding as in the case when  $\frac{4-2b}{d}>1$, we obtain  (\ref{long time 1}). Likewise, for  $\frac{1}{2}\le s<\min \left\{ \frac{d-2b}{2}+1,\frac{4-2b}{d}+1 \right\}$, fix  $s_0\in (s,\min \left\{ \frac{d-2b}{2}+1,\frac{4-2b}{d}+1 \right\})$. By Bernstein and Young's inequality,
	  \begin{align}
	  	f(N)&\lesssim _s \sigma _J(\frac{N}{2})+\sum _{M\le \eta N}(\frac{M}{N})^{s_0}f(M)+\left(\frac{K}{(\eta N)^{2s}}\right)^{\frac{2-b}{d}}   \|u\|^{\frac{4-2b}{d}}_{L_t^\infty \dot H^{s-\frac{1}{2}}_x}f(\eta N)^{1-\frac{4-2b}{d}}\notag\\ 
	  	 &\lesssim  _s \frac{1}{\eta ^{s}}\frac{K^{\frac{1}{2}}}{N^s} \|u\|_{L^\infty _t\dot H^{s-\frac{1}{2}}_x(J\times \mathbb{R} ^d)}+\sigma  _J(\frac{N}{2})+\sum _{M\le \eta N}(\frac{M}{N})^{s_0}f(M).\notag
	  \end{align}
	  Proceeding as in the case $\frac{4-2b}{d}>1$, we obtain  (\ref{long time 2}).
\end{proof}
The following estimate is an immediately consequence of  Proposition  \ref{T:long time SZ}, and will be used in Section \ref{S:6} to control the Morawetz action.  
\begin{cor}
		Under the hypotheses of   Proposition   \ref{T:long time SZ},   for any   $0<L\lesssim 1$, any admissible pair  $(q,r)$ with  $q>2$ and  $\frac{1}{q}<s\le1$,
	\begin{equation}
		\||\nabla |^sP_{\le LK}u\|_{L^q_tL_x^{r,2}(J\times \mathbb{R} ^d)}\lesssim _s  L^{s-\frac{1}{q}}K^s.\label{6105}
	\end{equation}
	\end{cor}
	\begin{proof}
	By long time Strichartz estimate (\ref{long time 1}) and interpolation
	\begin{align}
		&\||\nabla |^sP_{\le LK}u\|_{L^q_tL_x^{r,2}(J\times \mathbb{R} ^d)}\notag\\
		&\lesssim _s\sum _{N\le LK}N^s  \|P_NP_{\le LK}u\| _{L^q_tL_x^{r,2}(J\times \mathbb{R} ^d)}\lesssim _s\sum _{N\le LK}N^s  \|P_Nu\|_{L^q_tL_x^{r,2}(J\times \mathbb{R} ^d)}\notag\\
		&\lesssim _s\sum_{N\le LK} N^s \left((\frac{K}{N})^{\frac{1}{q}}+1\right)\lesssim L^{s-\frac{1}{q}}K^s.\notag
	\end{align}
\end{proof}

\section{Rapid cascade}\label{S:5}
In this section we exclude the first of two blow-up solutions, the rapid frequency cascade.
\begin{thm}\label{T:cascade}
	There does not exist  an almost periodic  solution to (\ref{nls}) when 
	\begin{equation}
		\int_0^\infty  N(t)^3dt<\infty. \label{6161}
	\end{equation} 
\end{thm}
The key to proving Theorem \ref{T:cascade} lies in using   Proposition  \ref{T:long time SZ} and (\ref{5181}) to demonstrate that in the case (\ref{6161}) the  almost periodic solution has additional regularity.
\begin{lem}\label{L:gain regularity}
	Suppose  $u$ is an almost periodic solution to (\ref{nls}), and  $\int_0^\infty N(t)^3dt=K<\infty $. Then for  $0\le s<\min \left\{ \frac{d-2b}{2}+1,\frac{4-2b}{d}+1 \right\}$, 
	\begin{equation}
		\|u(t,x)\|_{L^\infty _t\dot H^s_x((0,\infty )\times \mathbb{R} ^d)}\lesssim K^s.\notag 
	\end{equation}   
\end{lem}
\begin{proof}
		First by   (\ref{6161}) and (\ref{E:N}),   
	\begin{equation} 
		|N(T_1)-N(T_2)|\lesssim  \int _{T_2}^{T_1} |N'(t)|dt\lesssim \int _{T_2}^{T_1}N(t)^3dt\rightarrow0,\quad\text{as}\quad T_2\rightarrow\infty .\notag
	\end{equation} 
		 This together with  $\int_0^\infty N(t)^3dt<\infty $ implies 
	\begin{equation}
		\lim_{t\rightarrow+\infty }N(t)=0.\label{5183}
	\end{equation}
 Further utilizing	  (\ref{E:compact2}) and  (\ref{5183}), we have,   for any  $N>0$ and  any  $\eta>0$, 
	\begin{equation}
	\limsup_{T\rightarrow+\infty } \sigma _{[0,T]}\big(\tfrac{N}{2}\big)	\le \limsup _{T\rightarrow+\infty } \|P_{>N}u(T)\|_{L_x^2}\le \limsup _{t\rightarrow+\infty }  \|P_{>C(\eta) N(t)} u(t)\|_{L_x^2}  \le \eta.\notag
	\end{equation}
	Hence, for any  $N>0$
	\begin{equation}
			\limsup_{T\rightarrow+\infty } \sigma _{[0,T]}\big(\tfrac{N}{2}\big)	= \limsup _{T\rightarrow+\infty } \|P_{>N}u(T)\|_{L_x^2}=0.\notag
	\end{equation}   
Then by (\ref{long time 2}) and (\ref{5181}),   for  $\frac{1}{2}\le s<\min \left\{ \frac{d-2b}{2}+1,\frac{4-2b}{d}+1 \right\}$, 
	\begin{equation}
		 \|P_{>N}u\|_{L^2_tL_x^{\frac{2d}{d-2},2}((0,\infty )\times \mathbb{R} ^d)} + \|P_{>N}F(u)\|_{L^2_tL_x^{\frac{2d}{d+2},2}((0,\infty )\times \mathbb{R} ^d)}\lesssim  \frac{K^{\frac{1}{2}}}{N^{s}} \|u\|_{L^\infty _t \dot H^{s-\frac{1}{2}}_x((0,\infty )\times \mathbb{R} ^d)}. \notag
	\end{equation}
	Therefore, for any     $0\le t_0<\infty $ and  $ \frac{1}{2}\le s< \min \left\{ \frac{d-2b}{2}+1,\frac{4-2b}{d}+1 \right\}$,
\begin{align}
	 \|P_{>N}u(t_0)\|_{L^2_x}^2&\lesssim  |\int _{t_0}^\infty  \langle \frac{d}{dt}P_{>N}u,P_{>N}u\rangle dt|\lesssim  |\int _{t_0}^\infty  \langle P_{>N}F(u),P_{>N}u\rangle dt| \notag\\
	 &  \lesssim \|P_{>N}F(u)\|_{L^2_tL_x^{\frac{2d}{d+2},2}}  \|P_{>N}u\| _{L^2_tL_x^{\frac{2d}{d-2},2}}    \lesssim \frac{K}{N^{2s}} \|u\|^2_{L^\infty _t \dot H_x^{s-\frac{1}{2}}}.\notag 
\end{align}
Hence for any   $\gamma \le s<\min \left\{ \frac{d-2b}{2}+1,\frac{4-2b}{d}+1 \right\}$ with  $0\le \gamma <\frac{1}{2}$,  
	\begin{align}
		 \|P_{>K}u\|_{\dot H^s}&\lesssim  \sum _{M>K}M^s \|P_Mu\|_{L^2_x}\lesssim \sum_{M>K}M^{\gamma -\frac{1}{2}}K^{\frac{1}{2}} \|u\|_{L_t^\infty \dot H_x^{s-\gamma }} 
		 \lesssim   K^\gamma   \|u\|_{L^\infty _t\dot H^{s-\gamma }_x}.\label{5182}   
	\end{align}  
	Using Bernstein's inequality and  (\ref{5182}) to separately control the high and low frequencies, we get
	\begin{equation}
		 \|u\|_{L^\infty _t\dot H^s_x((0,\infty )\times \mathbb{R} ^d)}\lesssim _{s,\gamma }K^s+K^\gamma  \|u\|_{L^\infty _t\dot H^{s-\gamma }_x((0,\infty )\times \mathbb{R} ^d)}. \label{53w2}
	\end{equation}
	Iterating (\ref{53w2}) starting with  $s=\gamma $ and  using the  conservation of mass, we  obtain the desired estimate in   Lemma \ref{L:gain regularity}.   
\end{proof}
Now we utilize the additional regularity in Lema \ref{L:gain regularity} to establish Theorem \ref{T:cascade}.    
\begin{proof}[\textbf{Proof of Theorem \ref{T:cascade}.}]
By Lemma \ref{L:gain regularity}, for   $t>0,1<s< \min \left\{ \frac{d-2b}{2}+1,\frac{4-2b}{d}+1 \right\}$, 
	\begin{align}
		 \|u(t)\|_{\dot H^1}&\lesssim   \|u_{>C(\eta)N(t)}\|_{\dot H^1}+ \|u_{\le C(\eta)N(t)}\|_{\dot H^1}\notag\\
		 &\lesssim      \|u_{>C(\eta)N(t)}\|_{ L^2} ^{1-\frac{1}{s}} \|u\|_{\dot H^s}^{\frac{1}{s}}+C(\eta)N(t)\notag\\
		 &\lesssim  \eta ^{1-\frac{1}{s}} K+C(\eta)N(t).\notag
	\end{align} 
	Therefore, by (\ref{5183}),   
	\begin{equation}
		\limsup _{t\rightarrow\infty } \|u(t)\|_{\dot H^1}=0,
	\end{equation}
	which together with  the Gagliardo-Nirenberg's inequality (\ref{E:GN1}) implies that   $E(u(t))\rightarrow0$. Hence $u\equiv0$, which is a contradiction.   
\end{proof}

\section{Quasi-soliton}\label{S:6}
In this section we defeat the second almost periodic solution, the quasi-soliton. Exclusion of this scenario will complete the proof of Theorem \ref{T:1}.  

Let   $\phi$ be a radial nonegative decreasing function such that  
\begin{equation}
	\phi(x)=1, \ |x|\le 1 \qquad \text{supp } \phi \subset \{x:|x|\le2\}.\label{E:phi}
\end{equation}
Let 
\begin{equation}
	\psi_t (r)=:\frac{1}{r}\int _0^r \phi (\frac{sN(t)}{R})ds.\notag
\end{equation}
Then 
\begin{equation}
	\psi_t(r)\ge\phi (\frac{rN(t)}{R})\quad\text{and}\quad \frac{d}{dt}\left(N(t)\psi_t(r)\right)=N'(t)\phi (\frac{rN(t)}{R}).\label{5311}
\end{equation}
Moreover, since  $\psi_t(r)=1$ when  $r\le \frac{R}{N(t)}$, we have 
\begin{equation}
	r\psi_t'(r)=\phi (\frac{rN(t)}{R})-\psi_t(r)=0\quad\text{when}\quad r\le \frac{R}{N(t)}.\label{5184}
\end{equation}  
Let  $I=P_{\le K}$ and  $F(u)=\mu |x|^{-b}|u|^{\frac{4-2b}{d}}u,\mu=\pm1$.     Note that 
\begin{equation}
	i\partial_{t}Iu+\Delta Iu=F(Iu)+ IF(u)-F(Iu). \label{614w1}
\end{equation}
Define the Morawetz action 
\begin{equation}
	M(t)=:2\int \psi_t(|x|)x_jN(t)\text{Im}[\overline{Iu}(t,x)\partial_{j}Iu(t,x)]dx,\notag
\end{equation} 
where we sum over repeated indices by convention.
By direct computation, we have 
\begin{align}
	&\ \quad	\frac{d}{dt}M(t)\notag\\
		&= 4N(t) \int \psi_t (|x|)|\nabla Iu|^2dx +4N(t)\int \psi_t '(|x|)\frac{x_j x_k}{|x|}\text{Re} \overline{\partial_{j}Iu}\partial_{k}Iudx\label{531}\\
	&\quad +\frac{4d\mu}{d+2-b}\int \psi_t (|x|)N(t)|x|^{-b}|Iu|^{\frac{4-2b}{d} +2}dx\label{532} \\
	&\quad   +\frac{(4-2b)\mu }{d+2-b}\int N(t)\psi_t'(|x|)|x|^{1-b}|Iu|^{\frac{4-2b}{d} +2}dx  \label{5321}  \\
	&\quad  +2N'(t)\int  \phi (\frac{rN(t)}{R})x_j\text{Im}[\overline{Iu}(t,x)\partial_{j}Iu(t,x)]dx    \label{534}\\
	&\quad  +N(t)\int \psi_t (|x|)x_j\partial_{j}\partial_{k}^2|Iu(t,x)|^2dx\label{614w2}\\
	&\quad  +2N(t)\int \psi_t (|x|)x_j\{ IF(u)-F(Iu),Iu\}_p^j dx,\label{693}
\end{align}
where the momentum bracket  $\{f,g\}=:\text{Re}\{f\overline{\nabla g}-\overline{\nabla f}g\}$.    

By (\ref{5311}) and (\ref{5184}), 
\begin{eqnarray}
	(\ref{531})&=& 4N(t)\int \psi_t (|x|)|\nabla Iu|^2dx 
	+4N(t)\int \left [\phi (\frac{rN(t)}{R})-\psi_t(r)\right] |\frac{x}{|x|} \cdot \nabla Iu |^2dx\notag\\
	&=&4N(t)\int \phi (\frac{rN(t)}{R}) |\nabla Iu|^2dx 
	+4N(t)\int (\psi_t (r)-\phi (\frac{rN(t)}{R})) |\not\negmedspace\nabla Iu|dx \notag\\
	&\ge& 4N(t)\int \phi (\frac{rN(t)}{R})|\nabla Iu|^2dx,\notag
\end{eqnarray}
where the angular derivative    $\not\negmedspace\nabla u=:\nabla u-(\frac{x}{|x|}\cdot \nabla u)\nabla u$.  

Furthermore, by $\phi-\phi^2\ge0$ and (\ref{5184}),   
\begin{align}
	&(\ref{531})+(\ref{532})+(\ref{5321})\notag\\
	&\ge 4N(t)\left[\int \phi^2 (\frac{rN(t)}{R})|\nabla Iu|^2dx+\frac{d\mu}{d+2-b}\int \phi (\frac{rN(t)}{R})|x|^{-b}|Iu|^{\frac{4-2b}{d} +2}dx\right]\notag\\
	&\qquad+O\left(\int _{|x|>\frac{R}{N(t)}}N(t)|x|^{-b}|Iu|^{\frac{4-2b}{d} +2}dx\right)\notag\\
	&=8N(t)E(\phi (\frac{rN(t)}{R})Iu)++O\left(\int _{|x|>\frac{R}{N(t)}}N(t)|x|^{-b}|Iu|^{\frac{4-2b}{d} +2}dx\right)+O(\frac{N(t)^3}{R^2}).\notag
\end{align}
In the defocusing case, the kinetic energy is directly controlled by the energy; while in the focusing case with  $ \|u_0\|_{L^2}< \|Q\|_{L^2} $,  this follows from  the  Gagliardo-Nirenberg's inequality (\ref{E:GN1}).  In summary, we can find  $\eta>0$ such that
\begin{align}
		(\ref{531})+(\ref{532}) +(\ref{5321})
		\ge&2\eta N(t)\int \phi^2 (\frac{rN(t)}{R})|\nabla Iu|^{2}dx \notag\\
	&+O\left(\int _{|x|>\frac{R}{N(t)}} N(t)|x|^{-b}|Iu|^{\frac{4-2b}{d} +2}dx\right)+O(\frac{N(t)^3}{R^2}).\label{525w2}
\end{align}
Next, by Young's inequality and integrating by parts, 
\begin{align}
	|(\ref{534})+(\ref{614w2})|&\lesssim  |N'(t)|\int \phi(\frac{rN(t)}{R}) \frac{R}{N(t)}|Iu||\nabla Iu|dx+O\left(\frac{N(t)^3}{R^2}\right)\notag\\
&\le \eta 	N(t)\int \phi^2(\frac{rN(t)}{R})|\nabla Iu|^2dx+C(\eta )R^2\frac{(N'(t))^2}{N(t)^3}+O\left(\frac{N(t)^3}{R^2}\right).\notag
\end{align}

Combining the above estimates with  (\ref{E:N}), we get
 \begin{align}
 		\frac{d}{dt}M(t)\ge &\eta N(t)\int \phi^2 (\frac{rN(t)}{R})|\nabla Iu|^2dx+(\ref{693})\notag\\
 	&+O\left(\int _{|x|>\frac{R}{N(t)}} N(t)|x|^{-b}|Iu|^{\frac{4-2b}{d} +2}dx\right)+O(\frac{N(t)^3}{R^2})+C(\eta)R^2|N'(t)|.\notag
 \end{align}
Integrating the above inequality and using Lemma \ref{L:693}, we obtain, 
\begin{align}
	&\int_0^T\frac{d}{dt}M(t)dt+ o_R(1)\cdot K+o_K(1)\cdot RK+CR^2\int_0^T |N'(t)|dt+\int_0^T|(\ref{693})|dt\gtrsim K,\notag
\end{align}
for  $R,K>R_0$.  
Assuming that we have proved the following Claims: 
\begin{claim}\label{C:691}
	$|M(t)|=o(K)\cdot R$, where  $\frac{o(K)}{K}\rightarrow0$ as  $K\rightarrow\infty $. 
\end{claim}
\begin{claim}
	\label{C:692}
	$\int_0^T|(\ref{693})|dt=o_K(1)\cdot RK$, where  $o_K(1)\rightarrow0$    as  $K\rightarrow\infty $.  
\end{claim}
Then we have, for  $R,K>R_0$,  
\begin{equation}
	o(K)\cdot R +o_R(1)\cdot K + o_K(1)\cdot RK +CR^2\int_0^T |N'(t)|dt\gtrsim K.\label{525w3}
\end{equation}
\noindent \textbf{Smoothing algorithm:}Using the smoothing algorithm from \cite[Section 6.1]{Dodson4},    we can replace 
  $N(t)$  with a much more slowly varying  $\widetilde{N}(t)$, thereby better controlling the last term on the left side of  (\ref{525w3}). 
  More specifically,  we  can construct  a sequence  $\{N_m(t)\}_{m=1}^\infty $, which satisfies the following properties:\\
  (i)  $N_m(t)\approx _m N(t)$.\\
  (ii)   $|\frac{N_m'(t)}{N_m(t)^3}|\le \frac{N'(t)}{N(t)^3}$.\\
  (iii)  $\int_0^T|N_m'(t)|dt\lesssim  \frac{K}{m}$, for  $K>m$. \\
  For the convenience of readers, we provide the construction  of  $N_m(t)$ in the appendix.

  Therefore, letting 
  \begin{equation}
  		M_m(t)=:2\int \psi_t(|x|)x_jN_m(t)\text{Im}[\overline{Iu}(t,x)\partial_{j}Iu(t,x)]dx,\notag
  \end{equation}
  and using the same argument used to derive (\ref{525w3}), we obtain, 
  \begin{equation}
  		o(K)\cdot R +o_R(1)\cdot K + o_K(1)\cdot RK + R^2\frac{K}{m}\gtrsim  K,\notag
  \end{equation}
  for  $R$ and  $K$  sufficiently large.   Fixing  $R$ and  $m$  sufficiently large, and then letting  $K\rightarrow \infty $, we obtain the contradiction  $u\equiv0$.

It therefore remains to prove Claim \ref{C:691} and Claim \ref{C:692}. 
\begin{proof}[\textbf{Proof of Claim \ref{C:691}:}]
	By Bernstein and (\ref{E:compact2}),
	\begin{align}
		\|\nabla Iu\|_{L^\infty _tL_x^2}&\lesssim\|\nabla P_{\le  K^{\frac{9}{10}}}Iu\|_{L^\infty _tL_x^2}+ \|\nabla P_{>K^{\frac{9}{10}}}Iu\|_{L^\infty _tL_x^2}\notag\\
		&\lesssim K^{\frac{9}{10}} \|u_0\|_{L^2}+K \|P_{>K^{\frac{9}{10}}}u\|_{L^\infty _tL_x^2}=o(K). \label{69w1}
	\end{align}
(\ref{69w1}) combined with   $|\psi_t(r)|\lesssim \frac{R}{rN(t)}$ implies
	\begin{equation}
		|M(t)|\lesssim  R \|Iu\| _{L^\infty _tL_x^2} \|\nabla Iu\|_{L^\infty _tL_x^2}\lesssim R \|\nabla Iu\|_{L^\infty _tL_x^2}=o(K)R.\notag  
	\end{equation}
\end{proof}
\begin{proof}[\textbf{Proof of Claim \ref{C:692}.}]
	Intergrating by parts and using the fact
	\begin{equation}
		\psi_t(|x|)|x|N(t)\lesssim R \quad\text{and}\quad |\partial_{j}(\psi_t(|x|)x_j)|=|\psi_t'(|x|)|x|+d\psi_t(|x|)|\lesssim 1,\notag
	\end{equation}
	we have
	\begin{align}
	\int_0^T	|(\ref{693})|dt\lesssim  &R\int_0^TN(t)dt\int | F(Iu)-IF(u)||\nabla Iu|dx\label{694} \\
		&+\int _0^TN(t)dt\int | F(Iu)-IF(u)||Iu|dxdt.\label{695}
	\end{align}
	By  H\"older,  Lemma \ref{L:error} in the Appendix and (\ref{6105}),
	\begin{equation}
		|(\ref{694})|\lesssim  R \| F(Iu)-IF(u)\|_{L^2_tL_x^{\frac{2d}{d+2},2}}  \|\nabla Iu\|_{L_t^2L_x^{\frac{2d}{d-2},2}}\lesssim RK o_K(1).\notag 
	\end{equation}
	Let  $J_k$ be local constant interval.  By Lemma \ref{L:error}  in the Appendix, Lemma \ref{L:6131} and Lemma \ref{L:6132} 
	\begin{align}
		|(\ref{695})|&\lesssim \sum_kN(J_k)\int_{J_k\times \mathbb{R} ^d}|  F(Iu)-IF(u)| | Iu|dxdt\notag\\
		&\lesssim  \sum _kN(J_k) \| F(Iu)-IF(u)\|_{L^2_tL_x^{\frac{2d}{d+2},2}(J_k\times \mathbb{R} ^d)} \|Iu\|_{L^2_tL_x^{\frac{2d}{d-2},2}(J_k\times \mathbb{R} ^d)}\notag\\
		&\lesssim \sum_kN(J_k)o_K(1)\lesssim \int_0^TN(t)^3dt o_K(1)\lesssim Ko_K(1).\notag  
	\end{align}
	This completes the proof of Claim \ref{C:692}.  
\end{proof}

\appendix

\section{Smoothing algorithm and control of error term} \label{s:app}
In this appendix, we  introduce the smoothing algorithm and use  the long  time Strichartz estimate to control the errors that arise from frequency truncation in Section \ref{S:6}. 
\subsection{Smoothing algorithm}
Divide  $[0,\infty ) $ into consecutive intervals  $J_n=:[a_n,a_{n+1})$ (called \emph{small interval}) such that 
\begin{equation}
	 \|u\|_{L_t^{\gamma }L_x^{\rho,2}([a_n,a_{n+1})\times \mathbb{R} ^d)}=1,\notag 
\end{equation}
where  $(\gamma ,\rho)$ was defined by (\ref{E:gamma}).   
By Lemma \ref{L:Ndj}, there exists  $J_0<\infty $  such that for all  $t\in [a_n,a_{n+1})$, 
\begin{equation}
	\frac{N(a_n)}{J_0}\le N(t)\le J_0N(a_n),\notag
\end{equation}
Possibly after modifying the  $ C(\eta)$ in  (\ref{E:compact1}), (\ref{E:compact2}) by a constant,  we can choose   \( N(a_n) = J_0^{i_n} \),   \( i_n \in \mathbb{Z}_{\leq 0} \), i.e. 
\begin{equation}
	\frac{N(a_n)}{N(a_{n+1})} = 1, \quad J_0, \quad \text{or} \quad J_0^{-1}; \label{6291}
\end{equation}
and  the starting and ending points of the peaks and valleys, as defined below, coincide with the starting or ending points of the  intervals  $[a_n,a_{n+1})$.  
 \begin{defn}[Peaks and valleys]\

 	(1) (\emph{peak}): An interval  $[a,b)$ is called as a peak of length  $n$, if it satisfies:\\
 	(i)  $N(t)\equiv C, \forall t\in [a,b)$, and  $ \|u\|^\gamma _{L^\gamma _tL_x^{\rho,2}([a,b)\times \mathbb{R} ^d) }=n $;\\
 	(ii) If \([a_-, a]\), \([b, b_+]\) are the small intervals adjacent to \([a, b]\),  then \( N(a_-) < N(a) \), \( N(b_+) < N(b) \),  i.e. \( N(a_-) = \frac{N(a)}{J_0} \), \( N(b_+) = \frac{N(b)}{J_0} \). 
 	
 	(2)  (\emph{valley}): An interval  $[a,b)$ is called as a  valley of length  $n$, if it satisfies:\\
 	(i)  $N(t)\equiv C, \forall t\in [a,b)$, and  $ \|u\|^\gamma _{L^\gamma _tL_x^{\rho,2}([a,b)\times \mathbb{R} ^d) }=n $;\\
 	(ii) If \([a_-, a]\), \([b, b_+]\) are the small intervals adjacent to \([a, b]\),  then \( N(a_-) >N(a) \), \( N(b_+) > N(b) \). 
 	
 	(3) In particular,  If \([a_-, a]\) and \([a, a_+]\) are adjacent small intervals, and \( N(a) > N(a_-), N(a_+) \), then we call \(\{a\}\) a peak of length 0. Similarly, if \( N(a_-) > N(a), N(a_+) \), then we call \(\{a\}\) a valley of length zero.
 \end{defn}
 \begin{rem}\label{R:6291}
We label the peaks \( p_k \) and the valleys \( v_k \). Because \( N(0) = 1 \) and \( N(t) \leq 1 \) we start with a peak.  Moreover, by (\ref{6291}),    the peaks and valleys are alternate:  \( p_0, v_0, p_1, v_1, \ldots \).
\end{rem}
\begin{lem}\label{L:6291}
	\begin{equation}
		\int _0^T|N'(t)dt|\le 2\sum _{0<p_k<T}N(p_k)+2. \label{629x1}
	\end{equation}
\end{lem}
\begin{proof}
 Since  $N(t)$ is monotonic between each peak and valley, 
 \begin{equation*}
 	\begin{aligned}
 		&\int_{v_k}^{p_{k+1}} |N'(t)| \, dt = N(p_{k+1}) - N(v_k) \leq N(p_{k+1}), \\
 		&\int_{p_k}^{v_k} |N'(t)| \, dt = N(p_k) - N(v_k) \leq N(p_k),
 	\end{aligned}
 \end{equation*}
 The above estimates combined with the endpoint, yield  (\ref{629x1}).   
\end{proof}
  Now we describe an iterative algorithm to construct progressively less oscillatory  $N_m(t)$.  
  
  \noindent\textbf{Step 1:}  $N_0(t)=N(t)$.\\
  \noindent\textbf{Step 2:} If  $[a,b]$ is peak for  $N_0(t)$ and  $[a_-,a],[b,b_+]$ are the adjacent small intervals, let  $N_1(t)=N_0(a_-)=\frac{N(a)}{J_0}$ for  $t\in [a_-,b_+]$.  \\
  $\cdots \cdots\cdots\cdots\cdots\cdots $\\
  \noindent\textbf{Step m+2:} If  $[a,b]$ is peak for  $N_m(t)$ and  $[a_-,a],[b,b_+]$ are the adjacent small intervals, let  $N_{m+1}(t)=N_m(a_-)=\frac{N_m(a)}{J_0}$ for  $t\in [a_-,b_+]$.  
  \begin{defn}
  	If \([a_m, b_m)\) and \([a_{m+1}, b_{m+1})\) are peaks of \( N_m(t) \) and \( N_{m+1}(t) \) respectively, and \([a_m, b_m) \subset [a_{m+1}, b_{m+1})\), then \([a_m, b_m)\) is called the parent of \([a_{m+1}, b_{m+1})\).
  	\end{defn}
  	In   Section \ref{S:6}, we utilized  some properties of   $N_m(t)$. These are established in the following proposition.
  	\begin{prop}\ 
  		(i)  $N(t)\ge N_{m+1}(t)\ge \frac{N_{m}(t)}{J_0}\ge \cdots \ge \frac{N(t)}{J_0^{m+1}}$,
  		
  	(ii) For a given peak \([a_{m+1}, b_{m+1})\) in \( N_{m+1}(t) \), the   peak \([a_m, b_m)\) in the previous generation \( N_m(t) \) must either satisfy \([a_m, b_m) \subset [a_{m+1}, b_{m+1})\) or \([a_m, b_m) \cap [a_{m+1}, b_{m+1}) = \emptyset\), 
  	
  	(iii) Each peak of \( N_{m+1}(t) \) has at least one parent,
  	
  	(iv) $|\frac{N_m'(t)}{N_m(t)^3}|\le C$, $ C $ is independent of  $m$,
  	
  	(iv) Each peak \([a_m, b_m)\) of \( N_m(t) \) contains at least \( 2m \) small intervals,
  	
  	(v) For any  $K> m$,
  	\begin{equation}
  		K=\int_0^T N(t)^3dt \gtrsim m\int_0^T |N_m'(t)|dt.\label{629x2}
  	\end{equation}
  		\end{prop}
  		\begin{proof}
  			  (i)--(iii) follows directly from the construction of  $N_m(t)$, and (iv) follows from (\ref{E:N}) and the fact that  $N_m(t)=N(t)$ whenever  $N_m'(t)\neq0$. \\
  			  (iv)  In the   construction of  \( N_m (t)\), the two small intervals adjacent to each peak of \( N_{m-1}(t) \) are modified. Thus, each iteration adds at least $ 2 $ small intervals.  \\
  			  (v)  By Lemma \ref{L:6131} and (i), 
  			  \begin{equation}
  			  	K\gtrsim  \sum_{J_n\subset [0,T]}N(J_n)\ge \sum _k\sum _{J_n \subset P_k^m}N_m(J_n).\notag
  			  \end{equation}
  			  The above estimate combined with (iv) and Lemma \ref{R:6291} implies 
  			  \begin{equation}
  			  	K\gtrsim 2m \sum _{0<p_k^m<T}N_m(p_k^m)\ge m\int_0^T|N_m'(t)|dt-2m,\notag
  			  \end{equation}
  			  which yields (v) for  $K\ge m$. 
  		\end{proof}
 
\subsection{Controlling the errors that arise from frequency truncation}
\begin{lem}\label{L:error}
	Let   $ u $ be the almost periodic solution given by    Theorem \ref{T:reduction} and $F(u)=\pm |x|^{-b}|u|^{\frac{4-2b}{d}}u$.  Then 
	\begin{equation}
		\| 	P_{\le  K}F(u)-F(P_{\le K}u)\|_{L^2_tL_x^{\frac{2d}{d+2},2}([0,T]\times \mathbb{R} ^d)}\lesssim o_K(1),\notag 
	\end{equation}
	where  $o_K(1)\rightarrow0$ as  $K\rightarrow\infty $.  
\end{lem}
\begin{proof}
	Let  $\rho$ be some quantity to be specified later,  $\rho\le \frac{1}{2}$ and  $\rho\ge K^{-1/10}$.  Let  $u_{l}=:P_{\le \rho K}u$,  $v=:u-u_l=P_{\ge \rho K}u$. 
	 Then 
	\begin{align}
		&P_{\le K}F(u)-F(P_{\le K}u) 
		=P_{\le K}F(u_l)-F( u_l)\label{6101}\\
		&\ -\left[F(P_{\le K}u)- F(u_l))-P_{\le K}(G(u_l,P_{\le K}v)v)\right]\label{6102}\\
		&\ + P_{\le K}F(u)-P_{\le K}F(u_l)-P_{\le K}(G(u_l,P_{\le K}v)v),\label{6103}
	\end{align}
	where 
	\begin{equation}
G(f,g)h=:|x|^{-b}\int _0^1 \left(\frac{d+2-b}{d} |f +\theta g|^{\frac{4-2b}{d}} h +\frac{2-b}{d} |f+\theta g |^{\frac{4-2b}{d}-2} (f+\theta g)^2 \overline{h}\right)d\theta.\notag 
	\end{equation}
	By Bernstein, Lemma \ref{L:nonlinear estimate} and (\ref{6105}),
	\begin{equation}
		\|(\ref{6101})\|_{L^2_tL_x^{\frac{2d}{d+2},2}}\lesssim \frac{1}{K}  \|u_l\|_{L_t^\infty L_x^2}^{\frac{4-2b}{d}} \|\nabla u_l\|_{L_t^2L_x^{\frac{2d}{d-2},2}}+\frac{1}{K} \||\nabla |^{\frac{2}{\gamma }} u_l\|_{L_t^\gamma L_x^{\rho,2}}^{\frac{\gamma }{2}}\lesssim \rho^{\frac{1}{2}}.   \notag
	\end{equation}
	Next, note that 
	\begin{equation}
		(\ref{6102})=-[G(u_l,P_{\le K}v)(P_{\le K}v)-P_{\le K}(G(u_l,P_{\le K}v)v)].\notag
	\end{equation}
	Let $\phi$ be the multiplier used in (\ref{E:phi}) to define the frequency projection.
	Using the fundamental theorem of calculus when  $ |\xi_2|\ll |\xi_1|$, and  $|\phi(\xi)|\lesssim 1$ when  $|\xi_1|\lesssim |\xi_2|$, we have 
	\begin{equation}
		|\phi (\frac{\xi_1+\xi_2}{K})-\phi (\frac{\xi_1}{K})|\lesssim  \frac{|\xi_2|}{|\xi_1|}.\label{6106}
	\end{equation}
	Let  $\varepsilon >0$ sufficiently small and  
	\begin{equation}
		m(\xi_1,\xi_2)=:\left(\phi (\frac{\xi_1+\xi_2}{K})-\phi (\frac{\xi_1}{K})|\right)\frac{K^\varepsilon }{|\xi_2|^\varepsilon }.\notag
	\end{equation}
Note that  on the support of  $\phi (\frac{\xi_1+\xi_2}{K})-\phi (\frac{\xi_1}{K})$, if  $K\gg |\xi_2|$ then  $K\lesssim |\xi_1|$.    
Hence,  by 	  (\ref{6106}),   $m(\xi_1,\xi_2)$	
	is a Coifman-Meyer symbol.     Applying Lemma \ref{L:CF},  H\"older and (\ref{long time 1}),
	\begin{align}
		\|(\ref{6102})\|_{L_t^2L_x^{\frac{2d}{d+2},2}}&\lesssim  \frac{1}{K^\varepsilon } \|P_{\le K}v\|_{L^2_tL_x^{\frac{2d}{d-2},2}} \||\nabla |^\varepsilon  \left(|x|^{-b}\int_0^1 |u_l+\theta P_{\le K}v|^{\frac{4-2b}{d}}d\theta \right)\|_{L^\infty _tL_x^{\frac{d}{2},\infty }}\notag\\
		&\lesssim \frac{1}{K^\varepsilon }(1+\frac{K^{1/2}}{(\rho K)^{1/2}}) \int _0^1\||\nabla |^\varepsilon (|u_l+\theta P_{\le K}v|^{\frac{4-2b}{d}})\|_{L^\infty _tL_x^{\frac{d}{2-b}}}d\theta  .\notag   
	\end{align}
	\noindent \textbf{When  $\frac{4-2b}{d}<1$:} By Lemma \ref{L:Visan},
	\begin{align}
		&	\||\nabla |^\varepsilon (|u_l+\theta P_{\le K}v|^{\frac{4-2b}{d}})\|_{L^\infty _tL_x^{\frac{d}{2-b}}} \notag\\
		&\lesssim  \|u_l+\theta P_{\le K}v\|_{L^\infty _tL^2_x}^{\frac{4-2b}{d}-\frac{\varepsilon }{1-\varepsilon }} \||\nabla |^{1-\varepsilon }(u_l+\theta P_{\le K}v)\|_{L^\infty _tL_x^2}^{\frac{\varepsilon }{1-\varepsilon }} \notag\\
		&\lesssim  K^{\frac{9}{10}\varepsilon }+K^\varepsilon  \|P_{>K^{\frac{9}{10}}}u\|_{L^\infty _tL^2_x}^{\frac{\varepsilon }{1-\varepsilon }},\notag   
	\end{align}
	where we used 
\begin{align}
		\||\nabla |^{1-\varepsilon }u_l\|_{L^2_x}&\lesssim   \||\nabla |^{1-\varepsilon } P_{\le K^{\frac{9}{10}}}P_{\le \rho K}u\|_{L^2_x}+ \||\nabla |^{1-\varepsilon } P_{>K^{\frac{9}{10}}}P_{\le \rho K}u\|_{L^2_x}\notag\\
		&\lesssim K^{\frac{9}{10}(1-\varepsilon )}+K^{1-\varepsilon } \|P_{>K^{\frac{9}{10}}}u\|_{L^2_x}.\label{721} 
\end{align}

	\textbf{When  $\frac{4-2b}{d}\ge1:$} By Lemma \ref{L:leibnitz} and the same argument used to derive (\ref{721}),  
	\begin{align}
		&	\||\nabla |^\varepsilon (|u_l+\theta P_{\le K}v|^{\frac{4-2b}{d}})\|_{L^\infty _tL_x^{\frac{d}{2-b}}} \notag\\
		&\lesssim  \|u_l+\theta P_{\le K}v\|_{L^\infty _tL^2_x}^{\frac{4-2b}{d}} \||\nabla |^{\varepsilon }(u_l+\theta P_{\le K}v)\|_{L^\infty _tL_x^2}\lesssim  K^{\frac{9}{10}\varepsilon }+K^\varepsilon  \|P_{>K^{\frac{9}{10}}}u\|_{L^\infty _tL_x^2}. \notag
	\end{align}
	Finally, we  estimate (\ref{6103}).  	By  H\"older and the long time Strichartz estimate (\ref{long time 1}), 
	\begin{align}
		&\|(\ref{6103})\|_{L_t^2L_x^{\frac{2d}{d+2},2}}= \| P_{\le K}(G(u_l,v)v)-P_{\le K}\left(G(u_l,P_{\le K }v)v\right)\|_{L_t^2L_x^{\frac{2d}{d+2},2}}\notag\\
		&\lesssim   \|v\|_{L_t^2L_x^{\frac{2d}{d-2},2}} \left \| \int_0^1 \left(|u_l+\theta v|^{\frac{4-2b}{d}}-|u_l+\theta P_{\le K}v|^{\frac{4-2b}{d}}\right)d\theta\right  \|_{L^\infty _tL_x^{\frac{d}{2},2}}\notag\\
		&\lesssim 	 \|P_{\ge \rho K}u\|_{L^2_tL_x^{\frac{2d}{d-2},2}} (\| u_l\|_{L^\infty _tL^2_x}+\|P_{\ge \rho K}u\|_{L^\infty _tL^2_x})^{\frac{4-2b}{d}-1}\|P_{\ge K}u\|_{L^\infty _tL^2_x}\notag\\
		&\lesssim  \frac{1}{\rho^{1/2}} \|P_{\ge K}u\|_{L^\infty _tL^2_x}   ,\notag
	\end{align}
	when  $\frac{4-2b}{d}>1$.  
	Similarly, when  $\frac{4-2b}{d}\le1$, we have 
	\begin{equation}
		\|(\ref{6103})\|_{L_t^2L_x^{\frac{2d}{d+2},2}}\lesssim   \frac{1}{\rho^{1/2}} \|P_{\ge K}u\|_{L^\infty _tL^2_x} ^{\frac{4-2b}{d}}.\notag
	\end{equation}
	Combining the above estimates, we get 
	\begin{align}
		&  \|P_{\le K}F(u)-F(P_{\le K}u)\|_{L^2_tL_x^{\frac{2d}{d+2},2}}\notag\\
		&\lesssim \rho +\frac{1}{\rho^{1/2}}(K^{-\frac{1}{10}\varepsilon }+ \|P_{>K^{\frac{9}{10}}}u\|_{L^\infty _tL^2_x}^{\frac{\varepsilon }{1-\varepsilon }}) +\frac{1}{\rho^{1/2}} \|P_{\ge CK}u\|_{L^\infty _t L^2_x}^{\min \left\{ \frac{4-2b}{d},1 \right\}} .\notag
	\end{align}
	Fixing  $\rho>0$ sufficiently small, then letting  $K\rightarrow\infty $,  we complete the proof of  Lemma \ref{L:error}.   
\end{proof}

		
\vskip 0.5in
	\subsection*{Declarations}
	
$\bullet$ Conflict of interest: There is no conflict of interest.

$\bullet$ Data availability:
	The authors declare that the data supporting the findings of this study are available within the paper, its supplementary information files, and the National Tibetan Plateau Data Center.


\begin{thebibliography}{99}	
	\bibitem{Aloui}
	L. Aloui, S. Tayachi,
	\emph{Local well-posedness for the inhomogeneous nonlinear Schr\"odinger equation,}
	Discrete Contin. Dyn. Syst. \textbf{48} (2021), no. 11, 5409–5437.
	
	\bibitem{BV2007}
	P. B\'egout, A. Vargas,
	\emph{Mass concentration phenomena for the $L^2$-critical nonlinear Schr\"odinger equation,}
	Trans. Amer. Math. Soc. \textbf{359} (2007), 5257–5282.
	
	\bibitem{Bergh1976}
	J. Bergh,  J. Löfström,
	\emph{Interpolation Spaces: An Introduction},
	Springer-Verlag, Berlin, 1976.
	
	
\bibitem{Bourgain1999}
J. Bourgain,
\emph{Global well-posedness of defocusing 3D critical NLS in the radial case,}
J. Amer. Math. Soc. \textbf{12} (1999), 145--171.

 
	
	\bibitem{Cazenave}
	T. Cazenave,
	\emph{Semilinear Schr\"odinger Equations},
	 Courant Lecture Notes in Math.  10, Amer. Math. Soc. (2003). 
	 
	 \bibitem{Colliander2008}
	 J. Colliander, M. Keel, G. Staffilani, H. Takaoka, T. Tao,
	 \emph{Global well-posedness and scattering for the energy-critical nonlinear Schrödinger equation in $\mathbb{R}^3$,}
	 Ann. of Math. \textbf{167} (2008), 767--865.
	 
	 
	\bibitem{Cruz}
	D. Cruz-Uribe, V. Naibo,
	\emph{Kato-Ponce inequalities on weighted and variable Lebesgue spaces,}
	Differential Integral Equations \textbf{29} (2016), 801–836.
	
	\bibitem{Dao}
	N. A. Dao, J. I. D\'iaz, Q. H. Nguyen,
	\emph{Generalized Gagliardo-Nirenberg inequalities using Lorentz spaces, BMO, H\"older spaces and fractional Sobolev spaces,}
	Nonlinear Anal. \textbf{173} (2018), 146–153.
	
 
	\bibitem{DinhKe}
	V. D. Dinh, S. Keraani,
	\emph{Energy scattering for a class of inhomogeneous biharmonic nonlinear Schr\"odinger equations in low dimensions,}
	arXiv:2211.11824v2.
	
	\bibitem{Dodson1}
	B. Dodson,
	\emph{Global well-posedness and scattering for the defocusing, $L^2$-critical nonlinear Schr\"odinger equation when $d\ge3$,}
	J. Amer. Math. Soc. \textbf{25} (2012), 429–463.
	
	\bibitem{Dodson4}
	B. Dodson,
	\emph{Global well-posedness and scattering for the mass critical nonlinear Schr\"odinger equation with mass below the mass of the ground state,}
	Adv. Math. \textbf{285} (2015), 1589–1618.
	
	\bibitem{Dodson3}
	B. Dodson,
	\emph{Global well-posedness and scattering for the defocusing, $L^2$-critical nonlinear Schr\"odinger equation when $d=2$,}
	Duke Math. J. \textbf{165} (2016), 3435–3516.
	
	\bibitem{Dodson2}
	B. Dodson,
	\emph{Global well-posedness and scattering for the defocusing, $L^2$-critical nonlinear Schr\"odinger equation when $d=1$,}
	Amer. J. Math. \textbf{138} (2016), 531–569.
	
	\bibitem{Dodson2019}
	B. Dodson,
	\emph{Global well-posedness and scattering for the focusing, cubic Schrödinger equation in dimension $d = 4$,}
	Ann. Sci. Éc. Norm. Supér. (4) \textbf{52} (2019), no. 1, 139--180.
	
	
	\bibitem{Genoud2008}
	F. Genoud,
	\emph{Th\'eorie de bifurcation et de stabilit\'e pour une \'equation de Schr\"odinger avec une non-lin\'earit\'e compacte (in French),}
	PhD Thesis no. 4233, EPFL 2008.
	
	\bibitem{Genoud2010}
	F. Genoud,
	\emph{Bifurcation and stability of travelling waves in self-focusing planar waveguides,}
	Adv. Nonlinear Stud. \textbf{10} (2010), 357–400.
	
	\bibitem{Genoud2011}
	F. Genoud,
	\emph{A uniqueness result for $\Delta u - \lambda u + V(|x|)u^p = 0$ on $\mathbb{R}^2$,}
	Adv. Nonlinear Stud. \textbf{11} (2011), 483–491.
	
	\bibitem{Genoud2012}
	F. Genoud,
	\emph{An inhomogeneous, $L^2$-critical nonlinear Schrödinger equation,}
	Z. Anal. Anwend. \textbf{31} (2012), 283–290.
	
	\bibitem{GS2008}
	F. Genoud, C. A. Stuart,
	\emph{Schr\"odinger equations with a spatially decaying nonlinearity: existence and stability of standing waves,}
	Discrete Contin. Dyn. Syst. \textbf{21} (2008), 137–186.
	
	\bibitem{Gill2000}
	T. S. Gill,
	\emph{Optical guiding of laser beam in nonuniform plasma,}
	Pramana \textbf{55} (2000), 835–842.
	
	 
	
	\bibitem{Keel-Tao}
	M. Keel, T. Tao,
	\emph{Endpoint Strichartz estimates,}
	Amer. J. Math. \textbf{120} (1998), 955–980.
	
	\bibitem{Kenig2006}
	C. Kenig, F. Merle,
	\emph{Global well-posedness, scattering, and blowup for the energy-critical, focusing, non-linear Schrödinger equation in the radial case,}
	Invent. Math. \textbf{166} (2006), 645--675.
 
	\bibitem{KTV}
	R. Killip, T. Tao, M. Visan,
	\emph{The cubic nonlinear Schrödinger equation in two dimensions with radial data,}
	J. Eur. Math. Soc. (JEMS) \textbf{11} (2009), 1203–1258.
	
	\bibitem{KillipVisan2010}
	R. Killip, M. Visan,
	\emph{The focusing energy-critical nonlinear Schrödinger equation in dimensions five and higher,}
	Amer. J. Math. \textbf{132} (2010), no. 2, 361--424. MR 2654778 (2011e:35357).
	
	
	\bibitem{KV}
	R. Killip, M. Visan, 
	\emph{Nonlinear Schrödinger equations at critical regularity},
 in: Evolution Equations, in: Clay Math. Proc., vol. 17, Amer. Math. Soc., Providence, RI, 2013.
	
	\bibitem{KVZ}
	R. Killip, M. Visan, X. Zhang,
	\emph{The mass-critical nonlinear Schrödinger equation with radial data in dimensions three and higher,}
	Anal. PDE \textbf{1} (2008), 229–266.
	
	\bibitem{LT1994}
	C. Liu, V. Tripathi,
	\emph{Laser guiding in an axially nonuniform plasma channel,}
	Phys. Plasmas \textbf{1} (1994), 3100–3103.
	
 
	
	\bibitem{MMZ}
	C. Miao, J. Murphy, J. Zheng,
	\emph{Scattering for the non-radial inhomogeneous NLS,}
	Math. Res. Lett. \textbf{28} (2021), no. 5, 1481–1504.
	
	\bibitem{ONeil}
	R. O'Neil,
	\emph{Convolution operators and $L(p,q)$ spaces,}
	Duke Math. J. \textbf{30} (1963), 129–142.
 
 \bibitem{ReedSimon1975}
 M. Reed, B. Simon,
 \emph{Methods of Modern Mathematical Physics, Vol. II: Fourier Analysis, Self-Adjointness,}
 Academic Press, 1975.
 
 
 \bibitem{Ryckman2007}
 E. Ryckman, M. Visan,
 \emph{Global well-posedness and scattering for the defocusing energy-critical nonlinear Schrödinger equation in $\mathbb{R}^{1+4}$,}
 Amer. J. Math. \textbf{129} (2007), no. 1, 1--60.  
 
 \bibitem{Tao2006}
 T. Tao,
 \emph{Global well-posedness and scattering for the higher-dimensional energy-critical nonlinear Schrödinger equation for radial data,}
 New York J. Math. \textbf{11} (2005), 57--80.  
 
 
	\bibitem{TVZ}
	T. Tao, M. Visan, X. Zhang,
	\emph{Global well-posedness and scattering for the defocusing mass-critical nonlinear Schr\"odinger equation for radial data in high dimensions,}
	Duke Math. J. \textbf{140} (2007), 165–202.
	
	\bibitem{Toland1984}
	J. F. Toland,
	\emph{Uniqueness of positive solutions of some semilinear Sturm-Liouville problems on the half line,}
	Proc. Roy. Soc. Edinburgh Sect. A \textbf{97} (1984), 259–263.
	
	
	
	\bibitem{Visan2007}
	M. Visan,
	\emph{The defocusing energy-critical nonlinear Schr\"odinger equation in higher dimensions,}
	Duke Math. J. \textbf{138} (2007), 281–374.
	
	\bibitem{Wei2023}
	W. Wei, Y. Wang, Y. Ye,
	\emph{Gagliardo–Nirenberg Inequalities in Lorentz Type Spaces,}
	J. Fourier Anal. Appl. \textbf{29} (2023), Paper No. 35, 30 pp.  
	
	
	\bibitem{Yanagida1991}
	E. Yanagida,
	\emph{Uniqueness of positive radial solutions of $\Delta u+g(r)u+h(r)u^p=0$ in $\mathbb{R}^n$,}
	Arch. Ration. Mech. Anal. \textbf{115} (1991), 257–274.
	
\end{thebibliography}
\end{document}